\documentclass[12pt,reqno]{NumPDEsArticle}

\usepackage[utf8]{inputenc}
\usepackage[T1]{fontenc}
\usepackage[english]{babel}
\usepackage{csquotes}

\usepackage{NumPDEsMacros}
\usepackage{amsmath,amssymb}
\usepackage{enumitem}
\usepackage{pgfplots}
\usepackage{pgfplotstable}

\let\vec\boldsymbol

\newcommand{\Cred}{C_{\rm red}}
\newcommand{\Cgoal}{C_{\rm go}}
\newcommand{\Cest}{C_{\rm est}}
\newcommand{\rhochild}{\rho_{\rm child}}
\newcommand{\Cpatch}{C_{\rm patch}}
\newcommand{\Csize}{C_{\rm size}}
\newcommand{\Cbdd}{C_{\rm bdd}}
\newcommand{\Cdeff}{C_{\rm deff}}
\newcommand{\rhodeff}{\rho_{\rm deff}}

\newcommand{\coarse}{H}
\newcommand{\fine}{h}
\newcommand{\normalvec}{\boldsymbol{n}}
\DeclareMathOperator{\osc}{osc}

\newcommand{\headlineRd}{\texorpdfstring{$\R^d$}{R\^d}}

\newcommand{\qbox}[1]{\quad\mbox{#1}}
\newcommand{\qqbox}[1]{\quad\mbox{#1}\quad}
\newcommand{\Rge}{\mathbb{R}_{\ge0}}

\newcommand{\TTgood}{\TT^{+}}
\newcommand{\TTbad}{\TT^{-}}
\newcommand{\Omegabad}{\Omega^{-}}
\newcommand{\TTugly}{\TT^{\star}}

\newcommand{\TTlk}{\TT_{\ell_k}}

\definecolor{color1}{HTML}{d62728}
\definecolor{color2}{HTML}{1f77b4}
\definecolor{color3}{HTML}{2ca02c}
\definecolor{color4}{HTML}{ff7f0e}

\newtheorem{markingstrategy}{Marking Strategy}

\title{Plain convergence of goal-oriented adaptive FEM}

\author{Valentin Helml}

\author{Michael Innerberger}

\author{Dirk Praetorius}

\address{TU Wien, Institute of Analysis and Scientific Computing, Wiedner Hauptstr. 8--10/E101/4, 1040 Vienna, Austria}
\email{valentin.helml@asc.tuwien.ac.at}
\email{michael.innerberger@asc.tuwien.ac.at \quad \textrm{(corresponding author)}}
\email{dirk.praetorius@asc.tuwien.ac.at}

\keywords{Adaptivity, goal-oriented algorithm, convergence, finite element method}

\thanks{The authors thankfully acknowledge support by the Austrian Science Fund (FWF) through the doctoral school \emph{Dissipation and dispersion in nonlinear PDEs} (grant W1245), the SFB \emph{Taming complexity in partial differential systems} (grant SFB F65) and the standalone project \emph{Computational nonlinear PDEs} (grant P33216).
Furthermore, we thank one of the anonymous referees for several valuable comments and input which improved content and presentation of this manuscript.}

\begin{document}

\maketitle

\begin{abstract}
    We discuss goal-oriented adaptivity in the frame of conforming finite element methods and plain convergence of the related \textsl{a posteriori} error estimator for different general marking strategies.
    We present an abstract analysis for two different settings.
    First, we consider problems where a local discrete efficiency estimate holds.
    Second, we show plain convergence in a setting that relies only on structural properties of the error estimators, namely stability on non-refined elements as well as reduction on refined elements.
	In particular, the second setting does not require reliability and efficiency estimates.
	Numerical experiments underline our theoretical findings.
\end{abstract}


\section{Introduction}

\subsection{Goal-oriented adaptive FEM}

In the context of computational partial differential equations (PDEs), standard adaptivity aims to approximate some unknown exact solution $u$ in the energy norm. While results on optimal convergence rates of adaptive FEM (AFEM) with respect to number of degrees of freedom or the overall computational cost are essentially tailored to the Dörfler marking strategy~\cite{doerfler1996}  (see, e.g.,~\cite{bdd2004, stevenson2007, ckns2008, axioms} and~\cite{stevenson2007,ghps2021} for quasi-optimal computational cost), plain convergence is mathematically understood for a quite general class of marking strategies and mathematical settings~\cite{msv2008, siebert2011, gp2021}.
 
However, in many applications in science and engineering, one is only interested in a derived  functional value $G(u)$ (also called \emph{quantity of interest}) of the PDE solution $u$, rather than the solution $u$ as a whole.
Unlike standard adaptivity,  goal-oriented adaptivity thus aims to approximate only $G(u)$.
Because of its practical relevance, it has also attracted quite some attention in the mathematical literature; see, e.g.,~\cite{eehj1995,br2001,bangerth-rannacher,gs2002} for some seminal works. 
These works led to the development of the \emph{dual-weighted residual (DWR)} method, which uses an adaptive algorithm where mesh-refinement is steered by residual error estimators for the approximation of the primal PDE solution $u$ that are locally weighted by the approximation error for the solution $z$ to the dual PDE.
While the DWR method is easy to implement in practice (even for nonlinear problems), it lacks a thorough convergence analysis, since a practical implementation must replace the unknown exact dual solution $z$ by a computable approximation $z \approx \widetilde{z}$, but fails to control this consistency error.

As an alternative to the DWR method, the works~\cite{ms2009,bet2011,hp2016,fpz2016} employ a non-trivial combination of primal and dual error indicators for the marking step to drive down an overall goal error estimator with optimal algebraic convergence rates.
As far as mathematical convergence results are concerned, available results rely on variants of the Dörfler marking strategy; see, e.g.,~\cite{ms2009,bet2011,fpz2016} for linear convergence with optimal convergence rates with respect to the number of degrees of freedom and~\cite{bgip2021} for optimal convergence rates with respect to the overall computational cost. However, unlike AFEM, a thorough mathematical analysis of plain convergence of goal-oriented adaptive FEM (GOAFEM) is open and the subject of the present work.

\subsection{An abstract approach to plain convergence of GOAFEM}

Following the corresponding works on plain convergence of AFEM~\cite{msv2008, siebert2011, gp2021}, we consider an abstract GOAFEM strategy: 
The PDE is usually formulated on a domain  $\Omega \subset \R^d$ with dimension $d \in \N_{\ge1}$ and the solution is sought in a space $\XX$ of functions on $\Omega$.
The goal quantity $G(u)$ is modeled by a continuous goal functional $G \colon \XX \to \R$.
To efficiently compute the goal value, goal-oriented adaptive algorithms adaptively refine a triangulation $\TT_\ell$ of $\Omega$ and, thus, enrich a corresponding (not necessarily conforming) discrete space $\XX_\ell$.
On this space, discrete solutions $u_\ell, z_\ell \in \XX_\ell$ with corresponding refinement indicators $\eta_\ell(T), \zeta_\ell(T) \in \R$ for all $T \in \TT_\ell$ as well as a discrete goal $G_\ell \in \R$ are computed.
The \emph{primal} function $u_\ell$ acts as approximation to the sought solution $u \in \XX$.
The \emph{dual} function $z_\ell$, on the other hand, is only used to compute the refinement indicators $\zeta_\ell(T)$ and is usually chosen such that, by a duality argument, it satisfies the \textsl{a posteriori} goal error bound
\begin{equation}
\label{eq:intro-estimate}
	|G(u) - G_\ell|
	\lesssim
	\bigg[ \sum_{T \in \TT_\ell} \eta_\ell(T)^2 \bigg]^{1/2}
	\bigg[ \sum_{T \in \TT_\ell} \zeta_\ell(T)^2 \bigg]^{1/2}
	=: \eta_\ell \, \zeta_\ell,
\end{equation}
where $\lesssim$ denotes $\leq$ up to some generic multiplicative constant.
Since, in contrast to the unknown goal error $|G(u) - G_\ell|$, the goal error estimator $\eta_{\ell} \zeta_{\ell}$ is \emph{computable}, any adaptive algorithm can only aim to minimize $\eta_{\ell} \zeta_{\ell}$.

E.g., for linear elliptic PDEs with linear goal $G$ in the standard Lax--Milgram setting, $u, z \in \XX$ are the solutions of
\begin{equation*}
	a(u,v) = F(v)
	\quad \text{and} \quad
	a(v,z) = G(v)
	\quad \text{for all } v \in \XX,
\end{equation*}
respectively, where $a(\cdot, \cdot)$ is a continuous and elliptic bilinear form and $F$ is a bounded linear functional on $\XX$.
For some discrete subspace $\XX_\ell \subset \XX$, the solutions are approximated by discrete functions $u_\ell, z_\ell \in \XX_\ell$ solving
\begin{equation*}
	a(u_\ell,v_\ell) = F(v_\ell)
	\quad \text{and} \quad
	a(v_\ell,z_\ell) = G(v_\ell)
	\quad \text{for all } v_\ell \in \XX_\ell.
\end{equation*}
Then, the goal error can be estimated by the primal and dual energy errors, i.e.,
\begin{equation*}
	|G(u) - G(u_\ell)|
	=
	a(u - u_\ell, z)
	=
	a(u - u_\ell, z - z_\ell)
	\lesssim
	\norm{u-u_\ell}_\XX \norm{z-z_\ell}_\XX,
\end{equation*}
from which~\eqref{eq:intro-estimate} with $G_\ell := G(u_\ell)$ follows by reliability $\norm{u-u_\ell}_{\XX} \lesssim \eta_\ell$ and $\norm{z-z_\ell}_{\XX} \lesssim \zeta_\ell$ of the error estimators.

In algorithmic terms, the goal-adaptive process takes the following form.
\begin{algorithm}[Abstract goal-oriented algorithm]\label{alg:goafem}
	\textsc{Input:} Initial triangulation $\TT_0$. \newline
	\textsc{Loop:} For all $\ell=0,1,2,\dots$ do: 
	\begin{enumerate}[label={\rm(\roman*)}]
		\item \textbf{SOLVE:} Compute the discrete solutions $u_\ell, z_\ell \in \XX_\ell$.
		\item \textbf{ESTIMATE:} Compute refinement indicators $\eta_\ell(T)$ and $\zeta_\ell(T)$ for all $T\in\TT_\ell$.
		\item \textbf{MARK:} Determine a set $\MM_\ell \subseteq \TT_\ell$ of marked elements.
		\item \textbf{REFINE:} Compute $\TT_{\ell+1} := \refine(\TT_\ell, \MM_\ell)$.
	\end{enumerate}
	\textsc{Output:} Triangulations $\TT_\ell$ and goal error estimators $\eta_\ell \zeta_\ell$.
\end{algorithm}

The main difference to standard adaptive algorithms (see, e.g., the seminal contributions~\cite{bv1984,doerfler1996,stevenson2007} or~\cite{axioms,gp2021} for some recent overview articles) is the marking step in Algorithm~\ref{alg:goafem}(iii), which takes into account primal and dual refinement indicators.
The focus of the present work is thus to analyze and understand for which marking criteria at least plain convergence can be guaranteed. To this end,
we consider two classes of marking criteria that generalize existing marking criteria for goal-oriented algorithms~\cite{ms2009,bet2011,fpz2016}.
We then show that \textsl{a priori} convergence of the approximate functions $u_\ell, z_\ell$ to \emph{some} limits (which do not need to coincide with $u$ and $z$, respectively) implies plain convergence of the goal error estimator, i.e.,
\begin{equation*}
	\eta_\ell \, \zeta_\ell
	\to 0
	\quad
	\text{as } \ell \to \infty.
\end{equation*}
The proof of this result requires either the abstract framework of~\cite{msv2008} for mesh-refinement (marked elements are refined, elements are the union of their children, and some size restrictions apply to all generated elements) and the refinement indicators $\eta_\ell(T), \zeta_\ell(T)$ (mainly discrete local efficiency); or similar assumptions as in~\cite{gp2021} for mesh-refinement (marked elements are refined and elements are the union of their children) and the refinement indicators (stability on non-refined elements and reduction on refined elements).

In particular, the convergence analysis presented in this work significantly relaxes the assumptions needed to prove plain convergence of the goal error in previous works.
Furthermore, it generalizes the marking criteria presented in~\cite{ms2009,bet2011,fpz2016} in several directions, in particular, to incorporate marking criteria from~\cite{msv2008,siebert2011} and to encompass marking criteria for nonlinear problems~\cite{hp2016,bip2021,bbi+2022}, where available results are scarce~\cite{bip2021,bbi+2022}, even if only plain convergence is aimed for~\cite{hpz15}.

\subsection{Outline}

We begin in Section~\ref{sec:setting} by outlining our adaptive algorithm and all assumptions of the two abstract settings, which are subsequently motivated in Section~\ref{sec:motivation} by means of the standard Poisson model problem with linear and nonlinear goal functional.
In Section~\ref{sec:results}, we state all main results, namely plain convergence of all proposed variations of our abstract adaptive algorithm.
The proofs of these results are then given in Sections~\ref{sec:proofA}--\ref{sec:proofB}.
Finally, we give some concrete examples that fit into our abstract settings as well as numerical results in Section~\ref{sec:examples}.

\subsection{Notation}
Throughout, $a \lesssim b$ abbreviates $a \leq C \, b$ with an arbitrary constant $C>0$ that does not depend on the local mesh size.
Furthermore, we denote by $a \simeq b$ that there holds $a \lesssim b$ and $b \lesssim a$.


\section{Abstract setting}\label{sec:setting}

Let $\Omega \subset \R^d$, $d \in \N_{\ge 1}$ be a bounded Lipschitz domain with positive Lebesgue measure $|\Omega|>0$.
Furthermore, let $\XX$ be a normed space of functions defined on $\Omega$.
We are interested in the value $G(u)$ of some given continuous functional $G \colon \XX \to \R$ evaluated at the solution $u$ of some given PDE.

\subsection{Discretization and mesh-refinement}

We suppose that discrete primal and dual solution are based on a mesh of $\Omega$, i.e., a finite set $\TT_\coarse$ that satisfies that
\begin{itemize}
    \item all elements $T\in\TT_\coarse$ are compact subsets of $\overline{\Omega}$ with positive measure $|T|>0$;
    \item for all $T,T' \in \TT_\coarse$ with $T \neq T'$, it holds that $|T \cap T'| = 0$;
    \item $\TT_\coarse$ covers $\overline{\Omega}$, i.e., $\bigcup\TT_\coarse = \overline\Omega$.
\end{itemize}
For a mesh $\TT_\coarse$ of $\Omega$ and a set of marked elements $\MM_\coarse \subseteq \TT_\coarse$, we call $\TT_\fine := \refine(\TT_\coarse,\MM_\coarse)$ the \emph{refinement} of $\TT_\coarse$ with respect to $\MM_\coarse$, if
\begin{enumerate}[label={\bf(R\arabic*)}, ref={R\arabic*}]
    \item\label{assumption:refine1} 
    all marked elements are refined, i.e., $\MM_\coarse \subseteq \TT_\coarse \backslash \TT_\fine$;

    \item\label{assumption:refine2} 
    all elements $T\in\TT_\coarse$ are the union of their children, i.e.,
    \begin{equation*}
        T = \bigcup \TT_\fine |_T
        \quad
        \text{for all } T\in\TT_\coarse,
        \text{ where} \quad
        \TT_\fine |_T 
        :=
        \{T' \in \TT_\fine \mid T' \subseteq T\}.
    \end{equation*}
\end{enumerate}
In the following, we suppose that the refinement strategy is arbitrary but fixed, satisfying~\eqref{assumption:refine1}--\eqref{assumption:refine2}.
Furthermore, for any given mesh $\TT_\coarse$, we write $\TT_\fine \in \T(\TT_\coarse)$ if $\TT_\fine$ is obtained by a finite number of refinement steps, i.e., there exist $\TT_0, \TT_1, \ldots, \TT_N$ with $\TT_0 = \TT_\coarse$ and $\TT_N = \TT_\fine$, and $\MM_\ell \subseteq \TT_\ell$ with $\TT_{\ell+1} = \refine(\TT_\ell, \MM_\ell)$ for all $\ell = 0, \ldots, N-1$.
Finally, let $\TT_0$ be a fixed initial mesh and define the set of all admissible meshes as $\T:=\T(\TT_0)$.

For each mesh $\TT_\coarse \in \T$, let $\XX_\coarse \subset \XX$ be an associated discrete subspace.
Let further $u_\coarse \in \XX_\coarse$ denote a suitable approximation to $u$ and $z_\coarse \in \XX_\coarse$ some additional function such that there holds \textsl{a priori} convergence, i.e., for the sequence $(\TT_\ell)_{\ell \in \N} \subset \T$ of meshes generated by the adaptive Algorithm~\ref{alg:goafem}, there exist $u_\infty, z_\infty \in \XX$ such that 
\begin{equation}\label{eq:a_priori_konv}
	\norm{u_\infty - u_\ell}_\XX \to 0
	\quad \text{and} \quad
	\norm{z_\infty - z_\ell}_\XX \to 0
	\qquad
	\text{as } \ell \to \infty.
\end{equation}

\begin{remark}
Note that it is not required that $u_\infty$ and $z_\infty$ coincide with $u$ and $z$, respectively.
Rather, \textsl{a priori} convergence is a standard assumption for the convergence analysis of adaptive algorithms (see, e.g., \cite{bv1984, msv2008}) and usually relies on structural properties of the underlying discretization. For conforming FEM with nested spaces $\XX_\ell \subseteq \XX_{\ell+1}$ for all $\ell \in \N_0$, it follows from the C\'ea-type quasi-optimality of discrete solutions, and $u_\infty, z_\infty \in \XX_\infty := \overline{\bigcup_{\ell=0}^\infty \XX_\ell}$ are the respective Galerkin approximations of $u$ and $z$ on the ``discrete limit space'' $\XX_\infty$. However, we note that the analysis is not necessarily restricted to conforming FEM, but also applies to a wider range of non-conforming methods as long as \textsl{a~priori} convergence is satisfied; see~\cite{ks2021,ks2022} for recent results on \textsl{a~priori} convergence for non-conforming FEM.
In this case, $\XX$ is a larger superspace containing the discrete FEM spaces as well as the space of the continuous PDE setting.
\end{remark}

\subsection{Goal-oriented error estimation}

Suppose that one can compute refinement indicators $\eta_\coarse(T), \zeta_\coarse(T) \geq 0$ for all $T \in \TT_\coarse$.
For $\xi_\coarse \in \{ \eta_\coarse, \zeta_\coarse \}$, define
\begin{equation}\label{eq:def_ref-ind}
    \xi_\coarse := \xi_\coarse ( \TT_\coarse ),
    \quad \text{where} \quad
    \xi_\coarse ( \mathcal{U}_\coarse ) := \bigg[ \sum_{T \in \mathcal{U}_\coarse} \xi_\coarse (T)^2 \bigg]^{1/2}
    \quad \text{for all }
    \mathcal{U}_\coarse \subseteq \TT_\coarse.
\end{equation}
Let $G_\coarse \in \R$ be the computable discrete goal quantity.
We assume that there exists a constant $\Cgoal > 0$ such that the goal error can be estimated by the (global) refinement indicators through
\begin{equation}
\label{eq:goal-error-estimate}
	|G(u) - G_\coarse| \leq \Cgoal \, \eta_\coarse\zeta_\coarse
	\quad
	\text{for all } \TT_\coarse \in \T.
\end{equation}

\subsection{Marking strategy}

The marking strategies we employ combine primal and dual error indicators and use one of the following two marking strategies for refinement indicators $\mu_\coarse(T)$ and a marking parameter $0 < \theta \leq 1$:
\begin{itemize}
	\item A bulk chasing strategy, also known as \emph{Dörfler marking} \cite{doerfler1996}:
	Find a set $\MM_\coarse \subseteq \TT_\coarse$ of marked elements such that
	\begin{equation}
	\label{eq:doerfler}
		\theta \, \mu_\coarse^2
		\leq
		\mu_\coarse(\MM_\coarse)^2.
	\end{equation}
	
	\item A general marking criterion proposed in~\cite{msv2008}:
	With a fixed function
	\begin{equation}
	\label{eq:M-assumptions}
		M \colon \R_{\geq 0} \to \R_{\geq 0}
		\quad \text{that is continuous at 0 with} \quad
		M(0) = 0,
	\end{equation}
	find a set $\MM_\coarse \subseteq \TT_\coarse$ of marked elements that satisfies
	\begin{equation}
	\label{eq:msv-criterion}
		\max_{T \in \TT_\coarse \backslash \MM_\coarse} \mu_\coarse(T) 
		\leq
		M \big( \mu_\coarse(\MM_\coarse) \big).
	\end{equation}
\end{itemize}
For either strategy, convergence of standard AFEM is known in general frameworks~\cite{msv2008, siebert2011, gp2021}.

\begin{remark}
	Note that the general criterion~\eqref{eq:msv-criterion} encompasses the \emph{maximum criterion}
	\begin{equation}
	\label{eq:maximum-marking}
		\MM_\coarse
		:=
		\set[\big]{T \in \TT_\coarse}{\mu_\coarse(T) \geq (1-\theta) \max_{T' \in \TT_\coarse} \mu_\coarse(T')}
	\end{equation}
	as well as the \emph{equidistribution criterion}
	\begin{equation}
	\label{eq:equidistribution-marking}
		\MM_\coarse
		:=
		\set[\big]{T \in \TT_\coarse}{\mu_\coarse(T) \geq (1-\theta) \, \mu_\coarse / \#\TT_\coarse}
	\end{equation}
	by $M(t) := t$.
	Furthermore, also the Dörfler marking~\eqref{eq:doerfler} is included by $M(t) := \sqrt{(1-\theta) \theta^{-1}} \, t$; see~\cite{gp2021}.
\end{remark}

In the following, we propose a generalized marking strategy in the framework of GOAFEM, which generalizes~\eqref{eq:msv-criterion} from AFEM to GOAFEM.
\begin{markingstrategy}\label{alg:mark_a}
	\textsc{Input:} 
    Function $V \colon \R_{\ge 0}^2 \to \R_{\ge 0}$ that is continuous at $(0,0)$ with $V(0,0) = 0$ and refinement indicators $\eta_\coarse(T)$ and $\zeta_\coarse(T)$ for all $T \in \TT_\coarse$.
    \newline
    Choose $\MM_\coarse \subseteq \TT_\coarse$ such that
    \begin{equation}\label{eq:mark_a}
        \big[
            \max_{ T \in \TT_\coarse \backslash \MM_\coarse } \eta_\coarse (T) 
        \big]
        \big[ 
            \max_{ T \in \TT_\coarse \backslash \MM_\coarse } \zeta_\coarse (T)
        \big]
        \leq
        V \big(
            \eta_\coarse ( \MM_\coarse ), 
            \zeta_\coarse ( \MM_\coarse ) 
        \big).
    \end{equation}
	\textsc{Output:} Set of marked elements $\MM_\coarse$.
\end{markingstrategy}

\begin{remark}
\label{rem:mark_a_stronger}
	Marking Strategy~\ref{alg:mark_a} can be seen as a generalization of the GOAFEM marking strategies presented in the seminal works~\cite{ms2009,bet2011}.
	In the following, we show that both strategies (with the general marking criterion~\eqref{eq:msv-criterion} instead of the Dörfler marking~\eqref{eq:doerfler} used in \cite{ms2009,bet2011}) already imply marking with our Marking Strategy~\ref{alg:mark_a} along the sequence $(\TT_\ell)_{\ell \in \N_0}$ generated by Algorithm~\ref{alg:goafem}.
	To this end, let $M$ be a function that satisfies~\eqref{eq:M-assumptions} and is quasi-monotonously increasing, i.e., there exists $\Cmon > 0$ such that \begin{align}\label{eq230304a}
	M(s) \leq \Cmon \, M(t) 
	\quad \text{for all $s,t \in \Rge$ with $s \leq t$.}
\end{align}%
	Note that in both settings considered below, there exists a constant $\Cest > 0$ such that (see Lemma~\ref{lemma:msv-estimators-bounded} and Lemma~\ref{lemma:estimator-boundedness} below)
	\begin{equation}\label{eq230304b}
		\sup_{\ell \in \N_0} ( \eta_\ell + \zeta_\ell )
		\leq 
		\Cest 
		< 
		\infty.
	\end{equation}
	\begin{itemize}
		\item \textbf{First mark, then combine:}
		Let $\MM_\ell^\eta,  \MM_\ell^\zeta \subseteq \TT_\ell$ be sets with
		\begin{equation*}
			\max_{ T \in \TT_\ell \backslash \MM_\ell^\eta} \eta_\ell (T)
			\leq
			M \big( \eta_\ell ( \MM_\ell^\eta ) \big)
			\quad \text{and} \quad
			\max_{ T \in \TT_\ell \backslash \MM_\ell^\zeta} \zeta_\ell (T)
			\leq
			M \big( \zeta_\ell ( \MM_\ell^\zeta ) \big).
		\end{equation*}
		The set $\MM_\ell \subseteq \TT_\ell$ of marked elements is then chosen such that, for $\xi \in \{ \eta, \zeta \}$, there holds $\MM_\ell^\xi \subseteq \MM_\ell \subseteq \MM_\ell^\eta \cup \MM_\ell^\zeta$.
		By defining $V(x,y) := \Cest \Cmon \, \max \big\{ M(x), M(y) \big\}$, we have that
		\begin{align*}
			&\big[
				\max_{T \in \TT_\ell \backslash \MM_\ell} \eta_\ell (T)
			\big]
			\big[
				\max_{T \in \TT_\ell \backslash \MM_\ell} \zeta_\ell (T)
			\big]
			\eqreff{eq230304b}\leq
			\Cest \max_{T \in \TT_\ell \backslash \MM_\ell} \xi_\ell (T)
			\leq
			\Cest M \big( \xi_\ell ( \MM_\ell^\xi ) \big)
			\\& \qquad
			\eqreff{eq230304a}\leq
			\Cest \Cmon M \big( \xi_\ell ( \MM_\ell ) \big)
			\leq
			V \big(
				\eta_\ell ( \MM_\ell ),
				\zeta_\ell ( \MM_\ell )
			\big).
		\end{align*}
		Thus, $\MM_\ell$ satisfies~\eqref{eq:mark_a} from Marking Strategy~\ref{alg:mark_a}.
		Based on the Dörfler marking strategy~\eqref{eq:doerfler} for $\eta_\ell$ and $\zeta_\ell$, this GOAFEM marking strategy is proposed and analyzed in~\cite{ms2009} for the Poisson model problem, while the analysis is extended to general second-order linear elliptic PDEs in~\cite{fpz2016}.
		Moreover, based on the maximum criterion~\eqref{eq:maximum-marking}, this strategy is proposed and experimentally considered in~\cite{mvy2020}.
		Finally, such a marking criterion (with certain technical modifications) leads even to instance-optimal adaptive meshes for the 2D Poisson problem~\cite{ip2021}.
		
		\item \textbf{First combine, then mark:}
		With the weighted error estimator
		\begin{equation*}
			\rho_\ell(T)^2
			:=
			\eta_\ell(T)^2 \, \zeta_\ell^2 + \eta_\ell^2 \, \zeta_\ell(T)^2,
		\end{equation*}
		the set $\MM_\ell \subseteq \TT_\ell$ of marked elements is then chosen such that
		\begin{equation}\label{eq230304c}
			\max_{T \in \TT_\ell \backslash \MM_\ell} \rho_\ell(T)
			\leq
			M \big( \rho_\ell(\MM_\ell) \big).
		\end{equation}
		 By defining $V(x,y) := \Cmon \, M \big( \Cest \, \sqrt{x^2 + y^2} \big)$, we have that
		 \begin{align*}
			 &\big[
			 	\max_{T \in \TT_\ell \backslash \MM_\ell} \eta_\ell (T)
			 \big]
			 \big[
			 	\max_{T \in \TT_\ell \backslash \MM_\ell} \zeta_\ell (T)
			 \big]
			 \leq
			 \max_{T \in \TT_\ell \backslash \MM_\ell}
			 \big( \eta_\ell (T)^2 \zeta_\ell^2 + \eta_\ell^2 \zeta_\ell (T)^2 \big)^{1/2}
			 \\& \qquad
			 =
			 \max_{T \in \TT_\ell \backslash \MM_\ell} \rho_\ell(T)
			 \eqreff{eq230304c}\leq
			 \rho_\ell(\MM_\ell)
			 =
			 M \big( \big[ \eta_\ell (\MM_\ell)^2 \zeta_\ell^2 + \eta_\ell^2 \zeta_\ell (\MM_\ell)^2 \big]^{1/2} \big)
			 \\& \qquad
			 \stackrel{\eqref{eq230304a},\eqref{eq230304b}}\leq
			 \Cmon \, M \big( \Cest \big[ \eta_\ell (\MM_\ell)^2 + \zeta_\ell (\MM_\ell)^2 \big]^{1/2} \big)
			 =
			 V \big(
				 \eta_\ell ( \MM_\ell ),
				 \zeta_\ell ( \MM_\ell )
			 \big).
		 \end{align*}
		 Thus, $\MM_\ell$ satisfies~\eqref{eq:mark_a} from Marking Strategy~\ref{alg:mark_a}. 
		Based on the Dörfler marking strategy~\eqref{eq:doerfler} for $\rho_\ell$, this GOAFEM marking strategy is proposed in~\cite{bet2011} and analyzed in~\cite{fpz2016}.
	\end{itemize}     
\end{remark}

The second strategy first combines the refinement indicators and then marks elements based on these weighted refinement indicators using a reformulation of the Dörfler marking:
If $\mu_\coarse \neq 0$, the set $\MM_\coarse$ from Dörfler marking~\eqref{eq:doerfler} satisfies
\begin{equation*}
	0
	<
	\theta
	\leq
	\frac{\mu_\coarse(\MM_\coarse)^2}{\mu_\coarse^2}.
\end{equation*}
This observation is used in~\cite{bet2011} to find a common framework for the marking strategies outlined in Remark~\ref{rem:mark_a_stronger}.
We present this in a more general manner.

\begin{markingstrategy}\label{alg:mark_b}
	\textsc{Input:} Weighting function $W \colon \Rge^2 \to \Rge$ that satisfies
	\begin{equation}
	\label{eq:C-assumptions}
		W(x,y) \leq C_W \max\{x,y\}
		\quad \text{for all } x,y \in [0,1],
		\quad \text{and} \quad
		W(1,1) = 1,
	\end{equation}
	where $C_W > 0$ is a fixed constant, refinement parameter $\theta \in (0,1]$, triangulation $\TT_\coarse$ and refinement indicators $\eta_\coarse(T)$ and $\zeta_\coarse(T)$ for all $T \in \TT_\coarse$. \newline
	Choose $\MM_\coarse \subseteq \TT_\coarse$ such that 
	\begin{equation}\label{eq:mark_b}
		0
		<
		\theta
		\leq
		W \Big( \frac{\eta_\coarse (\MM_\coarse)^2}{\eta_\coarse^2},
		\frac{\zeta_\coarse(\MM_\coarse)^2}{\zeta_\coarse^2}\Big).
	\end{equation}
	\textsc{Output:} Set of marked elements $\MM_\coarse$.
\end{markingstrategy}

\begin{remark}\label{rem:estimators_nonzero}
	If there holds $\eta_{\ell_0} \zeta_{\ell_0} = 0$ for some minimal $\ell_0 \in \N_0$ in Algorithm~\ref{alg:goafem}, then, by the goal error estimate~\eqref{eq:goal-error-estimate}, it follows that there already holds $G(u_{\ell_0}) = G(u)$.
	Therefore, Algorithm~\ref{alg:goafem} can be terminated after step~{\normalfont\rmfamily(ii)} in the $\ell_0$-th iteration.
	In particular, this means that $\eta_\ell \zeta_\ell \neq 0$ for all $\ell < \ell_0$ and Marking Strategy~\ref{alg:mark_b} is well-defined for all steps executed by Algorithm~\ref{alg:goafem}.
	In this case, for formal reasons, we define the sequence $(\TT_\ell)_{\ell \in \N_0} \subset \T$ by $\TT_{\ell} := \TT_{\ell_0}$ for all $\ell \geq \ell_0$.
\end{remark}

\begin{remark}\label{rem:marking-strategies}
	{\normalfont\rmfamily (i)}
	We note that Marking Strategy~\ref{alg:mark_b} is indeed stronger than and implies Marking Strategy~\ref{alg:mark_a}:
	To see this, suppose that $\MM_\ell \subseteq \TT_\ell$ satisfies Marking Strategy~\ref{alg:mark_b} and define
	\begin{equation*}
		V(x,y)
		:=
		\Bigl( \frac{\Cest C_W}{\theta} \Bigr)^{1/2} \max\{x,y\}
	\end{equation*}
	with $\Cest$ from~\eqref{eq230304b} and $C_W$ from~\eqref{eq:C-assumptions}.
	Without loss of generality, suppose that there holds $\eta_\coarse(\MM_\coarse)^2 / \eta_\coarse^2 \leq \zeta_\coarse(\MM_\coarse)^2 / \zeta_\coarse^2$.
	Then, it follows that
	\begin{equation*}
		\theta
		\stackrel{\eqref{eq:mark_b}}{\leq}
		W \Bigl( \frac{\eta_\coarse(\MM_\coarse)^2}{\eta_\coarse^2}, \frac{\zeta_\coarse(\MM_\coarse)^2}{\zeta_\coarse^2} \Bigr)
		\stackrel{\eqref{eq:C-assumptions}}{\leq}
		C_W \max \Bigl\{ \frac{\eta_\coarse(\MM_\coarse)^2}{\eta_\coarse^2}, \frac{\zeta_\coarse(\MM_\coarse)^2}{\zeta_\coarse^2} \Bigr\}
		=
		C_W \frac{\zeta_\coarse(\MM_\coarse)^2}{\zeta_\coarse^2}.
	\end{equation*}
	Together with~\eqref{eq230304b} this yields
	\begin{equation*}
		\theta \, \eta_\coarse^2 \, \zeta_\coarse^2
		\eqreff{eq230304b}{\leq}
		C_W \Cest \, \zeta_\coarse(\MM_\coarse)^2
		\leq
		C_W \Cest \max\{ \eta_\coarse(\MM_\coarse), \zeta_\coarse(\MM_\coarse) \}^2.
	\end{equation*}
	The concrete definition of $V(x,y)$ then proves~\eqref{eq:mark_a}.
	
	{\normalfont\rmfamily (ii)}
	The reason for considering Marking Strategy~\ref{alg:mark_b} instead of only Marking Strategy~\ref{alg:mark_a} is that it allows to relax the assumptions on the mesh-refinement; see Theorem~\ref{thm:goafem_b} below.
\end{remark}

\subsection{Estimator properties}

For the properties required of the refinement indicators, we consider two different settings that are laid out in the following.
In the context of standard AFEM, the first one stems from~\cite{msv2008} and covers a wide range of problems.
The second one stems from~\cite{gp2021} and is essentially based on higher-level properties of the estimator as extracted in~\cite{axioms}.
While the second setting is usually easier to prove in practice, it is tailored to residual-based error estimators, since it implicitly exploits that the local refinement indicators are weighted by the local mesh-size.
We refer to Section~\ref{sec:motivation} for an illustrative example for these properties.

\subsubsection{\textbf{Setting I based on discrete local efficiency}}\label{subsec:msv-setting}
We present here the setting originally conceived in the work~\cite{msv2008} for plain convergence of standard AFEM, which essentially relies on local discrete efficiency; see~\eqref{assumption:ldeff} below.

For all $\TT_\coarse \in \T$ and $\omega \subseteq \overline{\Omega}$, we define the corresponding patch and its area as
\begin{equation}
\label{eq:def-patch}
	\TT_\coarse [ \omega ]
	:=
	\set[\big]{T \in \TT_\coarse}{\overline{\omega} \cap T \neq \emptyset}
	\quad \text{and} \quad
	\Omega_\coarse [ \omega ]
	:=
	\bigcup \TT_\coarse [ \omega ],
\end{equation}
respectively.
For all $\UU_\coarse \subseteq \TT_\coarse$, we abbreviate
\begin{equation*}
	\TT_\coarse [ \UU_\coarse ]
	:=
	\TT_\coarse \Big[ \bigcup \UU_\coarse \Big]
	\quad \text{and} \quad
	\Omega_\coarse [ \UU_\coarse ]
	:=
	\bigcup \TT_\coarse [ \UU_\coarse ].
\end{equation*}

In addition to the properties~\eqref{assumption:refine1}--\eqref{assumption:refine2}, we require the following assumptions on the mesh-refinement:
With constants $0 < \rhochild < 1$ and $\Cpatch, \Csize \geq 1$, for all $\TT_\coarse \in \T$, $\MM_\coarse \subseteq \TT_\coarse$, and $\TT_\fine := \refine \big( \TT_\coarse, \MM_\coarse \big)$, it holds that
\begin{enumerate}[label={\bf(R\arabic*)}, ref={R\arabic*}, start=3]
	\item \label{assumption:refine3}
	proper children are uniformly smaller than their parents, i.e.,
	\begin{equation*}
		|T'| < \rhochild |T|
		\quad \text{for all }
		T \in \TT_\coarse \backslash \TT_\fine
		\text{ and }
		T' \in \TT_\fine
		\text{ with }
		T' \subseteq T;
	\end{equation*}
	
	\item \label{assumption:refine4}
	the number of elements per patch is uniformly bounded, i.e.,
	\begin{equation*}
		\# \TT_\coarse [T]
		\leq
		\Cpatch
		\quad \text{for all }
		T \in \TT_\coarse;
	\end{equation*}
	
	\item \label{assumption:refine5}
	patch elements have comparable size, i.e.,
	\begin{equation*}
		|T| \leq \Csize \, |T'|
		\quad \text{for all }
		T \in \TT_\coarse
		\text{ and }
		T' \in \TT_\coarse[T].
	\end{equation*}
\end{enumerate}
We further assume that the space $\XX$ is equipped with a familiy $\set[\big]{\norm{\cdot}_{\XX(\omega)}}{\omega \subseteq \overline{\Omega} \text{ measurable}}$ of seminorms which
\begin{enumerate}[label={\bf(N\arabic*)}, ref={N\arabic*}]
	\item \label{assumption:norm1}
	are subadditive, i.e., for all measurable subsets $\omega_1, \omega_2 \subseteq \overline{\Omega}$ with $| \omega_1 \cap \omega_2 | = 0$, it holds that
	\begin{equation*}
		\norm{v}_{\XX(\omega_1)} + \norm{v}_{\XX(\omega_2)}
		\leq
		\norm{v}_{\XX(\omega_1 \cup \omega_2)}
		\quad \text{for all }
		v \in \XX;
	\end{equation*}
	
	\item \label{assumption:norm2}
	are absolutely continuous with respect to the measure $|\cdot|$, i.e.,
	\begin{equation*}
		\lim_{ |\omega| \to 0 } \norm{v}_{\XX(\omega)} = 0
		\quad \text{for all }
		v \in \XX;
	\end{equation*}
	
	\item \label{assumption:norm3}
	approximate the natural norm on $\XX$, i.e., $\norm{\cdot}_{\XX(\overline{\Omega})} = \norm{\cdot}_\XX$.
\end{enumerate}
We suppose that there exist an additional normed space $\DD$ that is also equipped with a familiy of seminorms that satisfy~\eqref{assumption:norm1}--\eqref{assumption:norm3}, and some fixed data terms $D^\eta, D^\zeta \in \DD$.
Finally, let $0 <\rhodeff < 1$ and $\Cdeff \ge 1$ be constants. 
For all $\TT_\coarse \in \T$ and $\UU_\coarse \subseteq \TT_\coarse$ we denote by $\T( \TT_\coarse, \UU_\coarse, \rhodeff )$ the set of all $\TT_\fine \in \T(\TT_\coarse)$ such that $|T''| < \rhodeff |T'|$ for all $T \in \UU_\coarse, T' \in \TT_\coarse [T]$ and $T'' \in \TT_\fine$ with $T''\subseteq T'$.
Then, suppose that there hold:
\begin{enumerate}[label={\bf(A\arabic*)}, ref={A\arabic*}]
	\item \label{assumption:ldeff}
	\textbf{Local discrete efficiency}, i.e., for $(v, \xi) \in \{ (u, \eta), (z, \zeta) \}$, all $\TT_\coarse \in \T$, all $T \in \TT_\coarse$, and $\TT_\fine \in \T( \TT_\coarse, \{T\}, \rhodeff)$, the corresponding discrete solutions satisfy
	\begin{equation*}
		\Cdeff^{-2} \xi_\coarse (T)^2
		\leq
		\norm{v_\fine - v_\coarse}_{\XX( \Omega_\coarse [T] )}^2
		+ \osc^\xi_\coarse ( \TT_\coarse [T] )^2;
	\end{equation*}
	
	\item \label{assumption:osc}
	\textbf{Boundedness of the oscillation term} $\osc^\xi_\coarse ( \TT_\coarse [T] ) \in \R_{\ge0}$ from~\eqref{assumption:ldeff}, i.e., for $(v, \xi) \in \{ (u, \eta), (z, \zeta) \}$, it holds that
	\begin{equation*}
		\osc^\xi_\coarse ( \TT_\coarse [T] )^2
		\le
		w(|T|)
		\big[ 
		\norm{v_\coarse}_{\XX( \Omega_\coarse [T] )}^2
		+ \norm{D^\xi}_{\DD( \Omega_\coarse [T] )}^2
		\big]
		\quad \text{for all }
		T \in \TT_\coarse
	\end{equation*}
	where 
	$w \colon \R_{\ge0} \to \R_{\ge0}$
	is a monotonously increasing weight function, that is continuous at 
	$0$ with $w(0) = 0$;
	
	\item \label{assumption:approx}
	\textbf{A natural approximation property}, i.e., for all $\varepsilon > 0$ and all $\TT_\coarse \in \T$ there exists $\TT_\fine \in \T( \TT_\coarse, \TT_\coarse, \rhodeff)$ with $\norm{u - u_\fine}_\XX + \norm{z - z_\fine}_\XX \leq \varepsilon$.
\end{enumerate}

\subsubsection{\textbf{Setting II based on high-level structural properties}}\label{subsec:msv-setting}

In contrast to the previous setting, we follow here the recent work~\cite{gp2021} and only require the structural assumptions of stability~\eqref{assumption:stab} and reduction~\eqref{assumption:red}.
To this end, let $0 < \qred < 1$ and $S \colon \Rge^2 \to \Rge$ be a function that is continuous at~$(0,0)$ with $S(0,0)=0$ as well as bounded on compact subsets of $\Rge^2$.
For all $\TT_\coarse \in \T$ and $\TT_\fine \in \T \big(\TT_\coarse\big)$, we assume:
\begin{enumerate}[label={\bf(B\arabic*)}, ref={B\arabic*}]
	\item \label{assumption:stab}
	\textbf{Stability on non-refined elements}, i.e., for $\xi \in \{ \eta, \zeta \}$ it holds that
	\begin{equation*}
		\xi_\fine \big( \TT_\fine \cap \TT_\coarse \big)
		\leq
		\xi_\coarse \big( \TT_\fine \cap \TT_\coarse  \big) + S\big( \norm{u_\fine - u_\coarse}_\XX, \norm{z_\fine - z_\coarse}_\XX  \big);
	\end{equation*}
	\item \label{assumption:red}
	\textbf{Reduction on refined elements}, i.e., for $\xi \in \{ \eta, \zeta \}$ it holds that
	\begin{align*}
		\xi_\fine \big( \TT_\fine \backslash \TT_\coarse \big)^2
		\leq
		\qred \, \xi_\coarse \big( \TT_\coarse \backslash \TT_\fine \big)^2 + S\big( \norm{u_\fine - u_\coarse}_\XX, \norm{z_\fine - z_\coarse}_\XX  \big)^2.
	\end{align*}
\end{enumerate}

For a visual guide to the assumptions in both settings, see Figure~\ref{fig:settings}.

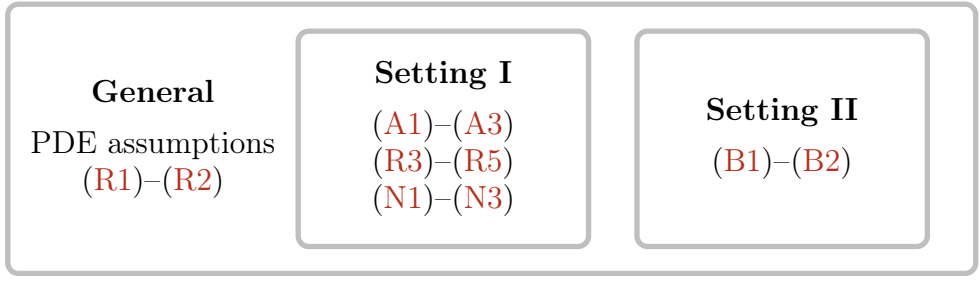
\begin{figure}
	\centering
	\begin{tikzpicture}[x=0.8\textwidth,y=0.15\textheight]
	\tikzstyle{category}=[line width=2pt, rounded corners, lightgray]
	\tikzstyle{categoryText}=[align=center]
	
	\draw[category] (0,0) rectangle (1,1);
	\node[categoryText] at (0.15,0.5) {\textbf{General} \\[1ex] PDE assumptions \\ \eqref{assumption:refine1}--\eqref{assumption:refine2}};
	
	\draw[category] (0.3,0.1) rectangle (0.6,0.9);
	\node[categoryText] at (0.45,0.5) {\textbf{Setting I} \\[1ex] \eqref{assumption:ldeff}--\eqref{assumption:approx}  \\ \eqref{assumption:refine3}--\eqref{assumption:refine5} \\ \eqref{assumption:norm1}--\eqref{assumption:norm3}};
	
	\draw[category] (0.65,0.1) rectangle (0.95,0.9);
	\node[categoryText] at (0.8,0.5) {\textbf{Setting II} \\[1ex] \eqref{assumption:stab}--\eqref{assumption:red}};
\end{tikzpicture}
	\caption{General scheme of assumptions made in this work.
	The assumptions about the PDE model as well as some of the assumptions on mesh-refinement (denoted by \enquote{R}) are general.
	Setting I requires further assumptions about mesh-refinement and involved norms (denoted by \enquote{N}).
	Finally, the assumptions on the estimator in Setting I and II are denoted by \enquote{A} and \enquote{B}, respectively.}
	\label{fig:settings}
\end{figure}


\section{Motivating example: Poisson equation in \headlineRd}\label{sec:motivation}

We clarify the parts of the abstract Algorithm~\ref{alg:goafem} and the assumptions made in the last section by an illustrative example.
Let $\Omega \subset \R^d$ with $d \in \N_{\geq 1}$ be a bounded Lipschitz domain with polygonal boundary $\partial \Omega$ and $\XX := H^1_0(\Omega)$ equipped with the \emph{energy norm} $\norm{\cdot}_\XX:= \norm{\nabla (\cdot)}_{L^2(\Omega)}$.
For $f\in L^2(\Omega)$, we consider the weak formulation of the Poisson equation:
\begin{equation}
\label{eq:poisson}
	\text{Find } u \in \XX: \quad 
	a(u,v)
	:= \int_\Omega \nabla u \cdot \nabla v \d{x}
	= \int_\Omega fv \d{x}
	=: F(v)
	\quad \text{for all } v \in \XX.
\end{equation}
We are interested in the goal value
\begin{equation}
\label{eq:poisson-goal}
	G(u)
	:=
	\int_\Omega g \, u \d{x}
	\quad
	\text{with } g \in L^2(\Omega).
\end{equation}
To this end, we introduce the dual problem:
\begin{equation}
\label{eq:poisson-dual}
	\text{Find } z \in \XX: \quad 
	a(v,z)
	= G(v)
	\quad \text{for all } v \in \XX.
\end{equation}
We note that, according to the Lax--Milgram lemma, there exist unique solutions $u \in \XX$ to~\eqref{eq:poisson} and $z \in \XX$ to~\eqref{eq:poisson-dual}.
Furthermore,~\eqref{eq:poisson-dual} allows to reconstruct the goal value by $G(u) = a(u,z) = F(z)$.

\subsection{Discretization}

For a given conforming triangulation $\TT_\coarse \in \T$ into simplices~\cite{stevenson2008}, we define the FEM space
\begin{equation*}
	\XX_\coarse
	=
	\SS_0^p(\TT_\coarse)
	:=
	\set[\big]{v_\coarse \in H^1_0(\Omega)}{v_\coarse|_T \text{ is a polynomial of degree $\leq p$ for all } T \in \TT_\coarse}.
\end{equation*}
For mesh-refinement, we employ newest vertex bisection \cite{stevenson2007,stevenson2008,bdd2004}, which satisfies~\eqref{assumption:refine1}--\eqref{assumption:refine2}.
Then, it holds that 
\begin{equation*}
    |T'| \leq |T|/2
    \quad \text{for all }
    \TT_\coarse \in \T,\,
    \TT_\fine \in \T(\TT_\coarse),\,
    T \in \TT_\coarse \backslash \TT_\fine
    \text{ and }
    T' \in \TT_\fine
    \text{ with }
    T'\subset T
\end{equation*}
and all triangulations are uniformly $\kappa$-shape regular for $\kappa > 0$ in the sense that
\begin{equation*}
	\max_{T \in \TT_\coarse} 
	\frac{ \text{diam}(T) }{ |T|^{1/d} }
	\le \kappa < \infty
	\quad \text{for all }
	\TT_\coarse \in \T,
\end{equation*}
which is why also~\eqref{assumption:refine3}--\eqref{assumption:refine5} hold.
The functions $u_\coarse$ and $z_\coarse$ are then obtained by discrete formulations of~\eqref{eq:poisson} and~\eqref{eq:poisson-dual}:
\begin{equation}
\label{eq:poisson-discrete}
\begin{split}
	\text{Find } u_\coarse, z_\coarse \in \XX_\coarse: \quad 
	a(u_\coarse, v_\coarse) &= F(v_\coarse), \\
	a(v_\coarse, z_\coarse) &= G(v_\coarse)
	\quad \text{for all } v_\coarse \in \XX_\coarse.
\end{split}
\end{equation}
Again, according to the Lax--Milgram lemma, there exists a unique solution $u_\coarse \in \XX_\coarse$ and $z_\coarse \in \XX_\coarse$, respectively.
We note that there holds the Galerkin orthogonality
\begin{equation}
\label{eq:galerkin-orthogonality}
	a(u-u_\coarse, v_\coarse)
	= 0 =
	a(v_\coarse, z-z_\coarse)
	\quad
	\text{for all } v_\coarse \in \XX_\coarse.
\end{equation}
Nestedness of the discrete function spaces, i.e., $\XX_\coarse \subseteq \XX_\fine$ if $\TT_\fine \in \T(\TT_\coarse)$, together with a C\'ea-type quasi-optimality result further imply that there holds \textsl{a priori} convergence~\eqref{eq:a_priori_konv}; see, e.g., the seminal work~\cite{bv1984}.

\subsection{Error estimators}

With $h_T := |T|^{1/d}$, local error indicators for~\eqref{eq:poisson-discrete} are given by 
\begin{equation}
\label{eq:poisson-estimators}
\begin{split}
	\mu_\coarse(T)^2
	&:=
	h_T^2 \, \norm{\Delta u_\coarse + f}_{L^2(T)}^2 + h_T \, \norm{\jump{\nabla u_\coarse \cdot \normalvec}}_{L^2(\partial T \cap \Omega)}^2,\\
	\nu_\coarse(T)^2
	&:=
	h_T^2 \, \norm{\Delta z_\coarse + g}_{L^2(T)}^2 + h_T \, \norm{\jump{\nabla z_\coarse \cdot \normalvec}}_{L^2(\partial T \cap \Omega)}^2.
\end{split}
\end{equation}

In the following we collect some basic properties of the presented error estimators~\cite{verfuerth}.
These are used to verify~\eqref{eq:goal-error-estimate} as well as~\eqref{assumption:ldeff}--\eqref{assumption:approx} and~\eqref{assumption:stab}--\eqref{assumption:red}.
\begin{proposition}[Reliability]\label{prop:poisson-rel}
	There exists a constant $\Crel > 0$ such that, for all $\TT_\coarse \in \T$, there holds
	\begin{equation}
	\label{eq:poisson-rel}
		\norm{\nabla (u - u_\coarse)}_{L^2(\Omega)}
		\leq
		\Crel \, \mu_\coarse
		\quad \text{and} \quad
		\norm{\nabla (z - z_\coarse)}_{L^2(\Omega)}
		\leq
		\Crel \, \nu_\coarse.
	\end{equation}
	The constant $\Crel$ depends only on $\kappa$-shape regularity.
\end{proposition}

Note that, for local discrete efficiency, the mesh-refinement rule (NVB) has to be modified in some cases (see \cite{egp2020} for details):
\begin{itemize}
	\item For $p=1$, all $T \in \MM_\coarse$ have to be refined by three levels of bisection to guarantee the presence of an interior vertex on all edges of $T$ and inside inside $T$~\cite{mns2000}. Then, there holds $\T(\TT_\coarse) = \T(\TT_\coarse, \MM_\coarse, \rhodeff)$ for all $\TT_\coarse \in \T$ and all $1/4 \leq \rhodeff < 1$.
	
	\item For $p>1$, the presence of an interior vertex inside $T$ is not required and standard NVB with three bisections per element (where all edges of $T$ are bisected once) suffices.
	Then, there holds $\T(\TT_\coarse) = \T(\TT_\coarse, \MM_\coarse, \rhodeff)$ for all $\TT_\coarse \in \T$ and all $1/4 \leq \rhodeff < 1$.
\end{itemize}
If $T$ is marked for refinement (and hence appropriately refined), these interior vertices give rise to discrete edge bubble functions which  are of polynomial order $1$ for refined edges and to element bubble functions of polynomial order $\max\{p-1, 1\}$ for marked elements. Together with the usual bubble function technique, this allows to prove to the following proposition:

\begin{proposition}[Local discrete efficiency]\label{prop:poisson-ldeff}
	There exists a constant $\Cdeff > 0$ such that, for all $\TT_\coarse \in \T$ and all $\TT_\fine \in \T(\refine(\TT_\coarse,\{T\}))$, there holds
	\begin{equation}
	\label{eq:poisson-ldeff}
	\begin{split}
		\Cdeff^{-2} \eta_\coarse(T)^2
		&\leq
		\norm{\nabla (u_\fine - u_\coarse)}_{L^2( \Omega_\coarse [T] )}^2
		+ \osc_\coarse(\TT_\coarse[T], f)^2,\\
		\Cdeff^{-2} \zeta_\coarse(T)^2
		&\leq
		\norm{\nabla (z_\fine - z_\coarse)}_{L^2( \Omega_\coarse [T] )}^2
		+ \osc_\coarse(\TT_\coarse[T], g)^2,
	\end{split}
	\end{equation}
	where the oscillation terms are defined in terms of the $L^2(T)$-orthogonal projection $\Pi_T^q$ onto polynomials of degree of at most $q$ on $T$ (and usually $q = p-1$ to ensure higher-order oscillations) and can be bounded by
	\begin{equation*}
		\osc_\coarse(\TT_\coarse[T], D)^2
		:=
		\sum_{T' \in \TT[T]} |T'|^{2/d} \norm{(1-\Pi_T^{q})D}_{L^2(T)}^2
		\leq
		\Cpatch \Csize^{2/d} |T'|^{2/d} \, \norm{D}_{L^2(T)}^2.
	\end{equation*}
	The constant $\Cdeff$ depends only on $\kappa$-shape regularity and the polynomial degree $p$.
\end{proposition}

\begin{proposition}[Stability]\label{prop:poisson-stab}
	There exists a constant $\Cstab > 0$ such that, for all ${\TT_\coarse \in \T}$ and $\TT_\fine \in \T(\TT_\coarse)$ as well as $\UU_\coarse \subseteq \TT_\fine \cap \TT_\coarse$, there holds that
	\begin{equation}
	\begin{split}
	\label{eq:poisson-stab}
		\mu_\fine( \UU_\coarse)
		&\leq
		\mu_\coarse( \UU_\coarse ) + \Cstab \, \norm{\nabla (u_\fine - u_\coarse)}_{L^2(\Omega)}\\
		\nu_\fine( \UU_\coarse )
		&\leq
		\nu_\coarse( \UU_\coarse ) + \Cstab \, \norm{\nabla (z_\fine - z_\coarse)}_{L^2(\Omega)}.
	\end{split}
	\end{equation}
	The constant $\Cstab$ depends only on $\kappa$ and the polynomial degree $p$.
\end{proposition}

\begin{proposition}[Reduction]\label{prop:poisson-reduction}
	There exist constants $\qred \in (0,1)$ and $\Cred > 0$ such that, for all $\TT_\coarse \in \T$ and $\TT_\fine \in \T(\TT_\coarse)$, there holds that
	\begin{equation}
	\label{eq:poisson-reduction}
	\begin{split}
		\mu_\fine( \TT_\fine \backslash \TT_\coarse )^2
		&\leq
		\qred \, \mu_\coarse( \TT_\coarse \backslash \TT_\fine )^2 + \Cred^2 \, \norm{\nabla (u_\fine - u_\coarse)}_{L^2(\Omega)}^2,\\
		\nu_\fine( \TT_\fine \backslash \TT_\coarse )^2
		&\leq
		\qred \, \nu_\coarse( \TT_\coarse \backslash \TT_\fine )^2 + \Cred^2 \, \norm{\nabla (z_\fine - z_\coarse)}_{L^2(\Omega)}^2
	\end{split}
	\end{equation}
	The constant $\Cred$ depends only on $\kappa$, $p$, and the element-size reduction, which is $1/2$ for NVB.
\end{proposition}

\begin{remark}
	If we set $S(r,t) := \max \{ \Cstab, \Cred \} \, \max \{r, t\}$, Propositions~\ref{prop:poisson-ldeff}--\ref{prop:poisson-reduction} together with \textsl{a priori} convergence~\eqref{eq:a_priori_konv} show that the error estimators $\mu_\coarse$ and $\nu_\coarse$ satisfy~\eqref{assumption:ldeff}--\eqref{assumption:approx} as well as ~\eqref{assumption:stab}--\eqref{assumption:red}.
	Furthermore, Galerkin orthogonality~\eqref{eq:galerkin-orthogonality} and Proposition~\ref{prop:poisson-rel} show that
	\begin{equation}
	\label{eq:poisson-goal-error}
	\begin{split}
		|G(u) - G(u_\coarse)|
		&\eqreff*{eq:poisson-goal}{=}
		a(u-u_\coarse,z)
		\eqreff*{eq:galerkin-orthogonality}{=}
		a(u-u_\coarse,z-z_\coarse)\\
		&\leq
		\norm{\nabla (u-u_\coarse)}_{L^2(\Omega)} \, \norm{\nabla (z-z_\coarse)}_{L^2(\Omega)}
		\eqreff{eq:poisson-rel}{\leq}
		\Crel^2 \, \mu_\coarse \, \nu_\coarse,
	\end{split}
	\end{equation}
	i.e., there holds~\eqref{eq:goal-error-estimate} with $\eta_\coarse := \mu_\coarse$ and $\zeta_\coarse := \nu_\coarse$.
	This motivates how the abstract settings from Section~\ref{sec:setting} relate to the present Poisson problem.
\end{remark}

\subsection{Quadratic goal functional}\label{subsec:quadratic goal}

Consider the quadratic goal value
\begin{equation}
\label{eq:quadratic-goal}
    G(u) := \int_\Omega g \, u^2 \d{x}
    \quad
    \text{with } g \in L^2(\Omega).
\end{equation}
The linearized dual problem for this goal functional and its discretization read:
\begin{align}
\label{eq:quadratic-dual}
	\text{Find } z \in \XX: \quad 
	a(v,z)
	&=  \scalarproduct{2g \, u}{v}_{L^2(\Omega)}
	\quad \text{for all } v \in \XX,\\
\nonumber
	\text{find } z_\coarse \in \XX_\coarse: \quad 
	a(v_\coarse,z_\coarse)
	&= \scalarproduct{2g \, u_\coarse}{v_\coarse}_{L^2(\Omega)}
	\quad \text{for all } v_\coarse \in \XX_\coarse.
\end{align}
Using this notion of dual problem, \cite{bip2021} shows the \textsl{a priori} goal error estimate
\begin{equation*}
\label{eq:quadratic-goal-error}
	|G(u) - G(u_\coarse)|
	\lesssim
	\norm{\nabla (u - u_\coarse)}_{L^2(\Omega)} \norm{\nabla (z - z_\coarse)}_{L^2(\Omega)}
			+ \norm{\nabla (u - u_\coarse)}_{L^2(\Omega)}^2.
\end{equation*}
This can be further estimated by reliability~\eqref{eq:poisson-rel} to obtain
\begin{equation*}
		|G(u) - G(u_\coarse)|
		\eqreff{eq:poisson-rel}{\lesssim}
		\eta_\coarse \, \zeta_\coarse + \eta_\coarse^2
		=
		\eta_\coarse \bigl( \zeta_\coarse + \eta_\coarse \bigr)
		\simeq
		\eta_\coarse \bigl( \zeta_\coarse^2 + \eta_\coarse^2 \bigr)^{1/2}.
\end{equation*}
Using $\eta_\coarse := \mu_\coarse$ and $\zeta_\coarse := (\mu_\coarse^2 + \nu_\coarse^2)^{1/2}$, this proves~\eqref{eq:goal-error-estimate}.
Furthermore, we note that $\zeta_\coarse$ also satisfies~\eqref{assumption:stab}--\eqref{assumption:red}; see~\cite{bip2021}.

\begin{remark}
	This choice of $\eta_\coarse$ and $\zeta_\coarse$ also appears for nonlinear primal problem; see, e.g., \cite{bbi+2022} for GOAFEM for semilinear PDEs.
\end{remark}

\begin{remark}
	By using that
	\begin{equation*}
		\mu_\coarse \, \nu_\coarse
		\leq
		\tfrac{1}{2} \, (\mu_\coarse^2 + \nu_\coarse^2)
		\quad \text{as well as} \quad
		\mu_\coarse \, (\mu_\coarse^2 + \nu_\coarse^2)^{1/2}
		\leq
		\mu_\coarse^2 + \nu_\coarse^2,
	\end{equation*}
	we see that the goal error of both linear and quadratic goal can also be estimated by
	\begin{equation*}
		|G(u) - G(u_\coarse)|
		\lesssim
		\mu_\coarse^2 + \nu_\coarse^2.
	\end{equation*}
	Thus, an algorithm for both linear and quadratic goal can also be realized by the goal error estimate~\eqref{eq:goal-error-estimate} with $\eta_\coarse := \zeta_\coarse := (\mu_\coarse^2 + \nu_\coarse^2)^{1/2}$.
\end{remark}

\subsection{Marking for the model problem}

For the Poisson problem with linear goal, \cite{ms2009} suggests to find sets $\MM_{\ell}^\eta$ and $\MM_{\ell}^\zeta$ such that
\begin{equation*}
	\theta \, \mu_\ell^2
	\leq
	\mu_\ell(\MM_{\ell}^\eta)^2
	\quad \text{and} \quad
	\theta \, \nu_\ell^2
	\leq
	\nu_\ell(\MM_{\ell}^\zeta)^2
\end{equation*}
and then choose $\MM_\ell$ as the smaller of the two.
This is covered by both our Marking Strategies~\ref{alg:mark_a} and~\ref{alg:mark_b} with $\eta_\ell := \mu_\ell$ and $\zeta_\ell := \nu_\ell$; see Remark~\ref{rem:mark_a_stronger} and~\cite{bet2011}.

For the quadratic goal, \cite{bip2021} suggests two marking stratgies:
The first one defines the weighted refinement indicator
\begin{equation*}
	\rho_\ell(T)^2
	:=
	\mu_\ell(T)^2 \big(\mu_\ell^2 + \nu_\ell^2 \big) + \mu_\ell^2 \big(\mu_\ell(T)^2 + \nu_\ell(T)^2 \big)
\end{equation*}
and obtains the set of marked elements $\MM_\ell$ by Dörfler marking~\eqref{eq:doerfler}.
This is covered by both our Marking Strategies~\ref{alg:mark_a} and~\ref{alg:mark_b} with $\eta_\ell := \mu_\ell$ and $\zeta_\ell:= \big( \mu_\ell^2 + \nu_\ell^2)^{1/2}$.
The second strategy presented determines the set $\MM_\ell$ by the modified Dörfler marking
\begin{equation*}
	\theta \, \big( \mu_\ell^2 + \nu_\ell^2 \big)
	\leq
	\mu_\ell(\MM_\ell)^2 + \nu_\ell(\MM_\ell)^2.
\end{equation*}
This is also covered by our Marking Strategies~\ref{alg:mark_a} and \ref{alg:mark_b} with $\eta_\ell := \zeta_\ell := \big( \mu_\ell^2 + \nu_\ell^2)^{1/2}$.

\subsection{GOAFEM with inexact solver}

Consider the case that the discrete solutions to~\eqref{eq:poisson-discrete} are computed by an iterative solver.
Suppose that one step of the solver on a mesh $\TT_\coarse \in \T$ is denoted by an operator $\Phi_\coarse^u, \Phi_\coarse^z \colon \XX_\coarse \to \XX_\coarse$ for the primal and dual problem, respectively, and that this solver is contractive, i.e., with the exact solutions $u_\coarse^\star, z_\coarse^\star$ of~\eqref{eq:poisson-discrete}, there exists a constant $0 < \kappa < 1$ such that, for all $v_\coarse \in \XX_\coarse$,
\begin{equation}
\label{eq:contractive-solver}
	\norm{u_\coarse^\star - \Phi_\coarse^u(v_\coarse)}_\XX
	\leq
	\kappa \, \norm{u_\coarse^\star - v_\coarse}_\XX
	\quad \text{and} \quad
	\norm{z_\coarse^\star - \Phi_\coarse^z(v_\coarse)}_\XX
	\leq
	\kappa \, \norm{z_\coarse^\star - v_\coarse}_\XX.
\end{equation}
Possible examples include the optimally preconditioned CG method (see \cite{cnx2012,ghps2021} for PCG and adaptive FEM) or geometric multigrid solvers (see \cite{wz2017,imps2022} in the frame of adaptive FEM).
On every level of the adaptive algorithm, we employ $k \geq 1$ steps of the iterative solver to obtain $u_\ell := (\Phi_\ell^u)^k(u_{\ell-1})$ and $z_\ell := (\Phi_\ell^z)^k(z_{\ell-1})$ with $u_{-1} = 0 = z_{-1}$, where $k = k(\ell)$ may even vary with $\ell$ (and could even differ for $u$ and $z$, respectively).

This setting is considered, e.g., in~\cite{fp2020,bgip2021}, and \cite{bgip2021} shows that the residual error estimators $\eta_\ell := \mu_\ell$ and $\zeta_\ell := \nu_\ell$ from~\eqref{eq:poisson-estimators} satisfy~\eqref{assumption:stab}--\eqref{assumption:red}.
Furthermore, contraction~\eqref{eq:contractive-solver} shows \textsl{a priori} convergence for the inexact solutions:
Since there holds \textsl{a priori} convergence~\eqref{eq:a_priori_konv} for the exact solutions $u_\ell^\star$, we have that
\begin{equation*}
	\norm{u^\star_{\ell+1} - u_{\ell+1}}_\XX
	\eqreff*{eq:contractive-solver}{\leq}
	\kappa^k \, \norm{u^\star_\ell - u_\ell}_\XX + \kappa^k \, \norm{u^\star_{\ell+1} - u^\star_\ell}_\XX
	\leq
	\kappa \, \norm{u^\star_\ell - u_\ell}_\XX + \norm{u^\star_{\ell+1} - u^\star_\ell}_\XX
\end{equation*}
is a contraction up to a null sequence.
Therefore, it follows that $\norm{u^\star_{\ell} - u_{\ell}}_\XX \to 0$ as $\ell \to \infty$; see Lemma~\ref{lemma:folgenkonv} below.
Thus, \textsl{a priori} convergence~\eqref{eq:a_priori_konv} for the exact solutions yields
\begin{equation*}
	\norm{u_\infty - u_{\ell}}_\XX
	\leq
	\norm{u_\infty - u^\star_{\ell}}_\XX + \norm{u^\star_{\ell} - u_{\ell}}_\XX
	\to 0
	\quad \text{as } \ell \to \infty;
\end{equation*}
the same argument shows $\norm{z_\infty - z_{\ell}}_\XX \to 0$ for the inexact dual solutions.
Finally, the goal error estimate reads (note that Galerkin orthogonality~\eqref{eq:galerkin-orthogonality} does not hold for inexact solutions)
\begin{align*}
	G(u) - G(u_\ell)
	=
	a(u - u_\ell, z)
	=
	a(u - u_\ell, z - z_\ell) + F(z_\ell) - a(u_\ell,z_\ell).
\end{align*}
With the discrete goal quantity $G_\ell := G(u_\ell) + F(z_\ell) - a(u_\ell,z_\ell)$, we obtain the goal error estimate~\cite{bgip2021}
\begin{equation*}
	|G(u) - G_\ell|
	\lesssim
	\big[ \eta_\ell +  \norm{u_\ell^\star - u_\ell}_\XX \big] \,
	\big[ \zeta_\ell +  \norm{z_\ell^\star - z_\ell}_\XX \big],
\end{equation*}
from which convergence of the goal error follows with \textsl{a priori} convergence of the inexact solutions (see above) as well as convergence of the estimator product $\eta_\ell \, \zeta_\ell$ and boundedness of the estimators (see Sections~\ref{sec:results}--\ref{sec:proofB} below).
The works~\cite{ghps2021,bgip2021} further present criteria to stop the iterative solver that allow to dominate the solver error $\norm{u_\ell^\star - u_\ell}_\XX$ by the estimator $\eta_\ell$ (and analogously for the dual problem); thus, one recovers the goal error estimate~\eqref{eq:goal-error-estimate}.

\begin{remark}
\label{rem:bem-fractional}
	Note that the explanatory example in this section, as well as the applied examples in Section~\ref{sec:examples}, deals only with finite element discretizations of (non-)linear second order elliptic PDEs.
	Indeed the motivation behind the subsequent analysis originates in conforming finite element methods, where adaptivity is driven by residual-type \textsl{a posteriori} error estimators.
	
	However, our analysis is more general than the presented examples suggest.
	In particular, it also covers the convergence analysis for point errors in adaptive boundary element methods~\cite{fgh+2016} and goal oriented adaptive FEM for problems featuring fractional diffusion; see, e.g., \cite{fmp2021} for the first convergence result for standard adaptive FEM for the fractional Laplacian.
\end{remark}


\section{Main results} \label{sec:results}

Our main results govern the (plain) convergence of Algorithm~\ref{alg:goafem} in both presented settings with Marking Strategies~\ref{alg:mark_a}--\ref{alg:mark_b} .

\subsection{Convergence in Setting I} \label{subsec:resultsA}

Our first main result states plain convergence in Setting I with Marking Strategy~\ref{alg:mark_a}.
This is an extension of the results of~\cite{msv2008} from standard AFEM to GOAFEM.

\begin{theorem}
	\label{thm:msv-result}
	Consider the output of Algorithm~\ref{alg:goafem} using Marking Strategy~\ref{alg:mark_a} or Marking Strategy~\ref{alg:mark_b}.
	Suppose that the refinement strategy satisfies~\eqref{assumption:refine1}--\eqref{assumption:refine5}.
	Suppose further that $\XX$ and $\DD$ are equipped with seminorms in the sense of~\eqref{assumption:norm1}--\eqref{assumption:norm3}.
	Finally, suppose that here holds local discrete efficiency as well as the natural approximation property~\eqref{assumption:ldeff}--\eqref{assumption:approx}.
	Then, \textsl{a priori} convergence~\eqref{eq:a_priori_konv} of $(u_\ell)_{\ell\in\N_0}$ and $(z_\ell)_{\ell\in\N_0}$ yields convergence
	\begin{equation*}
		\eta_\ell \, \zeta_\ell \to 0
		\quad
		\text{as } \ell \to \infty.
	\end{equation*}
\end{theorem}

\subsection{Convergence in Setting II} \label{subsec:resultsB}

The next results concern plain convergence in Setting II with both presented marking strategies.
The first one deals with Marking Strategy~\ref{alg:mark_a} and transfers~\cite[Theorem~3.1]{gp2021} from standard AFEM to GOAFEM.

\begin{theorem}\label{thm:goafem_a}
	Consider the output of Algorithm~\ref{alg:goafem} using Marking Strategy~\ref{alg:mark_a} or Marking Strategy~\ref{alg:mark_b}.
	Suppose that the error estimators $(\eta_\ell)_{\ell\in\N_0}$ and $(\zeta_\ell)_{\ell\in\N_0}$ satisfy stability~\eqref{assumption:stab} and reduction~\eqref{assumption:red}.
	Suppose further that the refinement strategy satisfies~\eqref{assumption:refine1}--\eqref{assumption:refine2}.
	Then, \textsl{a priori} convergence~\eqref{eq:a_priori_konv} of $(u_\ell)_{\ell\in\N_0}$ and $(z_\ell)_{\ell\in\N_0}$ yields convergence
	\begin{equation*}
		\eta_\ell \, \zeta_\ell \to 0
		\quad
		\text{as } \ell \to \infty.
	\end{equation*}
\end{theorem}

Similar to~\cite[Theorem~3.3]{gp2021} for Dörfler marking~\eqref{eq:doerfler}, our last main result exploits Marking Strategy~\ref{alg:mark_a}, which is inspired by Dörfler marking and allows to drop the child property~\eqref{assumption:refine2} from the list of assumptions.
Furthermore, we note that this theorem deals with the slightly more restrictive Marking Strategy~\ref{alg:mark_b}.

\begin{theorem}\label{thm:goafem_b} 
    Consider the output of Algorithm~\ref{alg:goafem} using Marking Strategy~\ref{alg:mark_b}. Suppose that the error estimators $(\eta_\ell)_{\ell\in\N_0}$ and $(\zeta_\ell)_{\ell\in\N_0}$ satisfy stability~\eqref{assumption:stab} and reduction~\eqref{assumption:red}.
    Suppose further 
    that the refinement strategy satisfies~\eqref{assumption:refine1}.
    Then, \textsl{a priori} convergence~\eqref{eq:a_priori_konv} of $(u_\ell)_{\ell\in\N_0}$ and $(z_\ell)_{\ell\in\N_0}$ yields convergence
    \begin{equation*}
	    \eta_\ell \, \zeta_\ell \to 0
	    \quad
	    \text{as } \ell \to \infty.
    \end{equation*}
\end{theorem}


\section{Proof of Theorem~\ref{thm:msv-result} (Setting I)} \label{sec:proofA}

The analysis presented in this section follows the work~\cite{msv2008}.
Without loss of generality, we suppose that Marking Strategy~\ref{alg:mark_a} is used; see Remark~\ref{rem:marking-strategies}.
We start with some auxiliary results.

\subsection{Auxiliary Results}\label{subsec:msv-auxiliary-results}
 
To abbreviate notation, let $(\xi, v) \in \{ (\eta, u), (\zeta, z) \}$.
For all $\ell \in \N_0$, define the following subsets of $\TT_\ell$:
\begin{align}
\begin{split}
    \TTgood_\ell
    &:=
    \set[\big]{T \in \TT_\ell}%
    	{\exists \ell' \ge \ell
            ~\forall T' \in \TT_\ell[T]
            ~\forall T'' \in \TT_{\ell'} \text{ with } T'' \subseteq T'
            \colon
            |T''| < \rhodeff |T'|};\\
    \TTbad_\ell
    &:=
    \set[\big]{T \in \TT_\ell}{\TT_\ell[T] \subseteq \TT_\infty};\\
    \TTugly_\ell
    &:=
    \TT_\ell \backslash \big[ \TTgood_\ell \cup \TTbad_\ell \big].
\end{split}
\end{align}
With this, it holds that
\begin{equation*}
    \xi_\ell^2
    =
    \xi_\ell ( \TTgood_\ell )^2
    + \xi_\ell( \TTugly_\ell )^2
    + \xi_\ell ( \TTbad_\ell )^2,
\end{equation*}
which allows us to investigate convergence of the corresponding error contributions separately.
As observed in~\cite{msv2008}, the contributions of $\TTgood_\ell$ and $\TTugly_\ell$ vanish in the limit.
To this end, we need \emph{local efficiency} which can be seen formally as~\eqref{assumption:ldeff} in the limit $v_\fine \to v$.
This, as well as local boundedness of the estimators, is stated in the next result.

\begin{proposition}
	\label{prop:local-efficiency}
	{\normalfont\rmfamily(i)}
	Suppose subadditivity~\eqref{assumption:norm1} of 
	$\norm{\cdot}_\XX$, local discrete efficiency~\eqref{assumption:ldeff}, and the natural approximation property~\eqref{assumption:approx}.
	Then, there holds local efficiency for $(\xi, v) \in \{ (\eta, u), (\zeta, z) \}$, i.e., for all $\TT_\coarse \in \T$, the corresponding discrete solution $v_\coarse$ satisfies
	\begin{equation}
	\label{eq:local-efficiency}
		\Cdeff^{-2} \xi_\coarse(T)^2
		\leq
		\norm{ v - v_\coarse }_{\XX(\Omega_\coarse[T])}^2
		+ \osc^\xi_\coarse ( \TT_\coarse[T] )^2
		\quad \text{for all }
		T \in \TT_\coarse.
	\end{equation}
	
	\noindent{\normalfont\rmfamily(ii)}
	Suppose that there hold local efficiency~\eqref{eq:local-efficiency} and the oscillation bound~\eqref{assumption:osc}.
	Then, with $\Cbdd := \Cdeff^2 ( 2 + w(|\Omega|) )$, there holds
	\begin{equation}
	\label{eq:local-boundedness}
		\Cbdd^{-1} \xi_\coarse(T)^2
		\leq
		\norm{v_\coarse}_{\XX(\Omega_\coarse[T])}^2
		+ \norm{v}_{\XX(\Omega_\coarse[T])}^2
		+ \norm{D^\xi}_{\DD(\Omega_\coarse[T])}^2.
	\end{equation}
\end{proposition}

\begin{proof}
	{\normalfont\rmfamily(i)}
	Let $\delta > 0$.
	Applying~\eqref{assumption:ldeff}, the Young inequality, subadditivity~\eqref{assumption:norm1} of $\norm{\cdot}_\XX$ yields that
	\begin{equation*}
	\Cdeff^{-2} \xi_\coarse(T)^2
	\le
	(1 + \delta) \norm{v - v_\coarse}^2_{\XX(\Omega_\coarse[T])}
	+
	(1 + \delta^{-1}) \norm{v - v_\fine}^2_{\XX}
	+
	\osc^\xi_\coarse ( \TT_\coarse[T] )^2
	\end{equation*}
	for all $\TT_\coarse \in \T$, $T \in \TT_\coarse$ and $\TT_\fine \in \T( \TT_\coarse, \{T\}, \rhodeff )$.
	By~\eqref{assumption:approx}, we can choose $\TT_\fine \in \T( \TT_\coarse, \TT_\coarse, \rhodeff )$ such that $\norm{v - v_\fine}^2_{\XX} \leq (1+\delta^{-1})^{-2}$.
	Since $\delta$ is arbitrary, we derive~\eqref{eq:local-efficiency}.
	
	\normalfont\rmfamily(ii)
	Together with the oscillation bound~\eqref{assumption:osc}, local efficiency~\eqref{eq:local-efficiency} implies that
	\begin{align*}
	\Cdeff^{-2} \, \xi_\coarse(T)^2
	&\eqreff*{eq:local-efficiency}{\le}
	\norm{v - v_\coarse }_{\XX(\Omega_\coarse[T])}^2
	+ \osc^\xi_\coarse ( \TT_\coarse[T] )^2
	\\
	&\eqreff*{assumption:osc}{\le}
	2 \norm{v}_{\XX(\Omega_\coarse[T])}^2
	+ \big( 2 + w(|T|) \big) \norm{v_\coarse}_{\XX(\Omega_\coarse[T])}^2
	+ w(|T|) \norm{D^\xi}_{\DD(\Omega_\coarse[T])}.
	\end{align*}
	This yields~\eqref{eq:local-boundedness}.
\end{proof}

\begin{proposition}[{\cite[Proposition 4.1]{msv2008}}]
\label{prop:good-convergence}
    Suppose that $\XX$ and $\DD$ are equipped with subadditive~\eqref{assumption:norm1} seminorms that satisfy~\eqref{assumption:norm3}.
    Suppose further that the refinement strategy satisfies~\eqref{assumption:refine2} and~\eqref{assumption:refine3}--\eqref{assumption:refine4}.
    Finally, suppose that there holds local discrete efficiency~\eqref{assumption:ldeff}--\eqref{assumption:osc}.
    Then, \textsl{a priori} convergence~\eqref{eq:a_priori_konv} yields
    \begin{equation}
    \label{eq:good-convergence}
        \xi_\ell ( \TTgood_\ell ) \to 0
        \quad \text{as }
        \ell \to \infty.
    \end{equation}
\end{proposition}

\begin{proposition}[{\cite[Proposition 4.2]{msv2008}}]
\label{prop:ugly-convergence}
    Suppose that $\XX$ and $\DD$ are equipped with subadditive~\eqref{assumption:norm1} and absolutely continuous~\eqref{assumption:norm2} seminorms.
    Suppose further that the refinement strategy satisfies~\eqref{assumption:refine2} and~\eqref{assumption:refine3}--\eqref{assumption:refine5}.
    Finally, suppose that there holds local efficiency~\eqref{eq:local-efficiency} as well as the oscillation bound~\eqref{assumption:osc}.
    Then, \textsl{a priori} convergence~\eqref{eq:a_priori_konv} yields
    \begin{equation}
    \label{eq:good-convergence}
        \xi_\ell ( \TTugly_\ell ) \to 0
        \quad \text{as }
        \ell \to \infty.
    \end{equation}
\end{proposition}

The following result is implicitly found in Step~2 of the proof of \cite[Lemma~4.1]{msv2008}.
We state it here explicitly, since it is used in our proofs of Lemma~\ref{lemma:msv-estimators-bounded} and Theorem~\ref{thm:msv-result}.

\begin{proposition}
\label{prop:sum-subadditive-seminorms}
    Let $\WW$ be a normed space equipped with subadditive seminorms~\eqref{assumption:norm1} and suppose that patches are bounded~\eqref{assumption:refine4}. Then, for all
    $\TT_\coarse \in \T$
    and
    $\UU_\coarse \subseteq \TT_\coarse$
    it holds that
    \begin{equation}
    \label{eq:sum-subadditive-seminorms}
        \sum_{T \in \UU_\coarse} \norm{w}_{\WW(\Omega_\coarse[T])}^2
        \le
        (\Cpatch^2 + 1) \norm{w}_{\WW(\Omega_\coarse[\UU_\coarse])}^2
        \quad \text{for all }
        w \in \WW.
    \end{equation}
\end{proposition}

Together, these last two results give boundedness of the estimators.

\begin{lemma}
\label{lemma:msv-estimators-bounded}
    Suppose that $\XX$ and $\DD$ are equipped with seminorms that satisfy~\eqref{assumption:norm1} and~\eqref{assumption:norm3}.
    Suppose further that patches are bounded~\eqref{assumption:refine4} and that there hold local efficiency~\eqref{eq:local-efficiency}
    as well as the oscillation bound~\eqref{assumption:osc}.
    Then, \textsl{a priori} convergence~\eqref{eq:a_priori_konv} yields
    \begin{equation}
    \label{eq:msv-estimators-bounded}
        C_\xi := \sup_{\ell \in \N_0} \xi_\ell < \infty.
    \end{equation}
\end{lemma}

\begin{proof}
    Proposition~\ref{prop:local-efficiency}{\normalfont\rmfamily(ii)} guarantees local boundedness~\eqref{eq:local-boundedness} of the estimator.
    With this, Propostion~\ref{prop:sum-subadditive-seminorms} and~\eqref{assumption:norm3} yield, for all $\ell \in \N_0$, that
    \begin{align*}
        \xi_\ell^2
        &=
        \sum_{T \in \TT_\ell} \xi_\ell (T)^2
        \eqreff*{eq:local-boundedness}{\le}
        \Cbdd \sum_{T \in \TT_\ell} \big(
            \norm{v_\ell}_{\XX(\Omega_\ell[T])}^2
            + \norm{v}_{\XX(\Omega_\ell[T])}^2
            + \norm{D^\xi}_{\DD(\Omega_\ell[T])}^2
        \big)
        \\
        &\eqreff*{eq:sum-subadditive-seminorms}{\le}
        \Cbdd \big( \Cpatch^2 + 1 \big) \big(
            \norm{v_\ell}_{\XX}^2
            + \norm{v}_{\XX}^2
            + \norm{D^\xi}_{\DD}^2
        \big).
    \end{align*}
    This, together with \textsl{a priori} convergence~\eqref{eq:a_priori_konv} of
    $(v_\ell)_{\ell \in \N_0}$
    guarantees that
    \begin{equation*}
        \xi_\ell^2
        \le
        \Cbdd \big( \Cpatch^2 + 1 \big) \big(
            \sup_{\ell \in \N_0} \norm{v_\ell}_{\XX}^2
            + \norm{v}_{\XX}^2
            + \norm{D^\xi}_{\DD}^2
        \big)
        =: C_\xi 
        \eqreff*{eq:a_priori_konv}{<}
        \infty,
    \end{equation*}
    which concludes the proof.
\end{proof}

\begin{proposition}[{\cite[Lemma 4.1]{msv2008}}]
\label{prop:T-infty-patch}
    Suppose that the refinement strategy satisfies~\eqref{assumption:refine2},~\eqref{assumption:refine3} and~\eqref{assumption:refine5}.
    Define the set
    \begin{equation}\label{eq:def_t_infty}
    	\TT_\infty := \bigcup_{\ell'\in\N_0} \bigcap_{\ell\geq\ell'} \TT_\ell
    \end{equation}
    of all elements $T \in \bigcup_{\ell \in \N_0} \TT_\ell$ which are not refined after finitely many steps.
    Let $T \in \TT_\infty$ be arbitrary.
    Then, there exists $\ell_0 \in \N_0$ such that
    \begin{equation}
        \TT_\ell [ T ]
        =
        \TT_{\ell_0} [ T ]
        \subseteq
        \TT_\infty
        \quad \text{for all }
        \ell \geq \ell_0.
    \end{equation}
\end{proposition}

The next result states uniform convergence of the sequence of local mesh-size functions
\begin{equation}
\label{eq:def-local-meshsize}
    h_\ell |_T := |T|^{1/d}
    \quad \text{for all }
    \ell \in \N_0
    \text{ and }
    T \in \TT_\ell.
\end{equation}
Note that $h_\ell\in L^\infty (\Omega)$ is well-defined, since the skeleton
\begin{equation}
\label{eq:def-skeleton}
    \Sigma_\ell
    :=
    \bigcup \big\{ \partial T \colon T \in \TT_\ell \big\}
\end{equation}
has measure zero for all $\ell \in \N_0$.

\begin{proposition}[{\cite[Lemma 4.3]{msv2008}}]
\label{prop:local-meshsize-convergence}
    Assume that the refinement strategy satisfies~\eqref{assumption:refine2} and~\eqref{assumption:refine3}.
    Then, there exists a function
    $h_\infty \in L^\infty(\Omega)$
    such that
    \begin{equation}
    \label{eq:local-meshsize-convergence}
        \norm{ h_\infty - h_\ell }_{L^\infty(\Omega)} \to 0
        \quad \text{as }
        \ell \to \infty.
    \end{equation}
\end{proposition}

\subsection{Proof of Theorem~\ref{thm:msv-result}} \label{subsec:msv-proof}
We split the proof of Theorem~\ref{thm:msv-result} into three lemmata.

\begin{lemma}
\label{lemma:good-ugly-convergence}
    Suppose the assumptions of Theorem~\ref{thm:msv-result}.
    Then, \textsl{a priori} convergence yields
    \begin{equation}
        \eta_\ell ( \TTgood_\ell \cup \TTugly_\ell )^2 \,
        \zeta_\ell ( \TTgood_\ell \cup \TTugly_\ell )^2
        \to 0
        \quad \text{as }
        \ell \to \infty.
    \end{equation}
\end{lemma}

\begin{proof}
    This is a direct consequence of Propositions~\ref{prop:good-convergence}--\ref{prop:ugly-convergence}.
\end{proof}

\begin{lemma}
\label{lemma:msv-proof-mixed-terms}
    Suppose the assumptions of Theorem~\ref{thm:msv-result}.
    Then, \textsl{a priori} convergence yields
    \begin{equation}
        \zeta_\ell ( \TTbad_\ell )^2
        \eta_\ell ( \TTgood_\ell \cup \TTugly_\ell )^2 
        + \eta_\ell ( \TTbad_\ell )^2
        \zeta_\ell ( \TTgood_\ell \cup \TTugly_\ell )^2 
        \to 0 
        \quad \text{as }
        \ell \to \infty.
    \end{equation}
\end{lemma}

\begin{proof}
    Propositions~\ref{prop:good-convergence} and~\ref{prop:ugly-convergence} yield that
    \begin{equation}
    \label{eq:mixed-terms-good-convergence}
        \eta_\ell ( \TTgood_\ell \cup \TTugly_\ell )^2
        + \zeta_\ell ( \TTgood_\ell \cup \TTugly_\ell )^2
        \to 0 
        \quad \text{as }
        \ell \to \infty.
    \end{equation}
    With Lemma~\ref{lemma:msv-estimators-bounded}, it follows that
    \begin{align*}
        \eta_\ell ( \TTbad_\ell )^2
        \zeta_\ell ( \TTgood_\ell \cup \TTugly_\ell )^2
        + \zeta_\ell ( \TTbad_\ell )^2
        \eta_\ell ( \TTgood_\ell \cup \TTugly_\ell )^2&
        \\
        \eqreff*{eq:msv-estimators-bounded}{\le}
        C_\eta^2 \, \zeta_\ell ( \TTgood_\ell \cup \TTugly_\ell )^2
        + C_\zeta^2 \, \eta_\ell ( \TTgood_\ell \cup \TTugly_\ell )^2 &
        \to 0
        \quad \text{as }
        \ell \to \infty.
    \end{align*}
    This concludes the proof.
\end{proof}

\begin{lemma}
\label{lemma:msv-bad-convergence}
    Suppose the assumptions of Theorem~\ref{thm:msv-result}.
    Then, \textsl{a priori} convergence yields
    \begin{equation}
        \zeta_\ell ( \TTbad_\ell )^2
        \eta_\ell ( \TTbad_\ell )^2
        \to 0 
        \quad \text{as }
        \ell \to \infty.
    \end{equation}
\end{lemma}

\begin{proof}
    Note that Proposition~\ref{prop:T-infty-patch} implies $\TT_\infty = \bigcup_{\ell \in \N_0} \TTbad_\ell$.
    With $\TT_\infty$ from~\eqref{eq:def_t_infty}, let
    \begin{equation*}
        \Omegabad := \bigcup \TT_\infty
        \quad \text{and} \quad
        \Sigma_\infty := \bigcup_{\ell \in \N_0} \Sigma_\ell.
    \end{equation*}
    For $x \in \Omegabad \backslash \Sigma_\infty$, let $T_x \in \TT_\infty$ be the unique element with $x \in T_x$ and choose the minimal index $\ell_x \in \N_0$ such that $T_x \in \TTbad_{\ell_x}$.
    For $\ell \in \N_0$ define $f_\ell \in L^1(\Omegabad \times \Omegabad)$ by
    \begin{equation*}
        f_\ell (x,y) :=
        \begin{cases}
            \frac{1}{|T_x|} 
            \eta_{\ell} (T_x) ^2
            \frac{1}{|T_y|} 
            \zeta_{\ell} (T_y) ^2
            & \text{if } 
            x,y \in \Omegabad \backslash \Sigma_\infty
            \text{ and }
            \ell \ge \ell_x,
            \\
            0                                   
            & \text{else}.
        \end{cases}
    \end{equation*}
    Then, we have the integral representation
    \begin{equation*}
        \int_{\Omegabad} \int_{\Omegabad} f_\ell(x,y) \d x \d y
        =
        \eta_\ell ( \TTbad_\ell )^2
        \zeta_\ell ( \TTbad_\ell )^2 
        \quad \text{for all }
        \ell \in \N_0.
    \end{equation*}
    Moreover, note that $f(x,y)>0$ yields that $T_x, T_y \in \TT_\ell \backslash \MM_\ell$.
    Property~\eqref{assumption:refine1} of the refinement strategy guarantees
    $\MM_\ell \subseteq \TTgood_\ell \cup \TTugly_\ell$
    for all
    $\ell \in \N_0$
    and Propositions~\ref{prop:good-convergence}--\ref{prop:ugly-convergence} imply that
    \begin{equation*}
        \eta_\ell ( \MM_\ell ) 
        + \zeta_\ell ( \MM_\ell )
        \to 0
        \quad \text{as }
        \ell \to \infty.
    \end{equation*}
    The marking criterion~\eqref{eq:mark_a} and continuity of $V$ at $(0,0)$ yield pointwise convergence
    \begin{equation*}
        0 
        \le 
        f_\ell (x,y) 
        \eqreff*{eq:mark_a}{\le}
        \frac{1}{|T_x| |T_y|} 
        V \big( \eta_\ell \big( \MM_\ell \big), \zeta_\ell \big( \MM_\ell \big) \big)^2
        \to 0
        \quad \text{as }
        \ell \to \infty
        \text{ for all }
        x,y \in \Omegabad \backslash \Sigma_\infty.
    \end{equation*}
    Recall the constants
    $C_\eta, C_\zeta > 0$
    from Lemma~\ref{lemma:msv-estimators-bounded}.
    For all
    $\ell \in \N_0$
    define
    \begin{equation*}
        g (x,y) :=
        \frac{1}{|T_x|}
        C_\eta^2
        \frac{1}{|T_y|}
        C_\zeta^2
        \quad \text{for all }
        x,y \in \Omegabad \backslash \Sigma_\infty.
    \end{equation*}
    Then, Lemma~\ref{lemma:msv-estimators-bounded} yields $|f_\ell(x,y)| \leq g_\ell(x,y)$ for all $\ell \in \N_0$ and almost all $x,y \in \Omegabad$.
    Thus, by the Lebesgue dominated convergence Theorem~\cite[p.1015]{zeidler1990}, we get
    \begin{equation*}
        \eta_\ell ( \TTbad_\ell )^2
        \zeta_\ell ( \TTbad_\ell )^2
        =
        \int_{\Omegabad} \int_{\Omegabad} f_\ell(x,y) \d x \d y
        \to 0
        \quad \text{as }
        \ell \to \infty,
    \end{equation*}
    which concludes the proof.
\end{proof}


\section{Proofs of Theorems~\ref{thm:goafem_a}--\ref{thm:goafem_b} (Setting II)}\label{sec:proofB}

\subsection{Auxiliary Results}

We start by showing uniform boundedness of the estimator.
To this end, let $(\xi, v) \in \{ (\eta, u), (\zeta, z) \}$.

\begin{lemma}\label{lemma:estimator-boundedness}
	Suppose that the error estimator $\xi_\ell$ satisfies stability~\eqref{assumption:stab} and reduction~\eqref{assumption:red}, and that there holds \textsl{a priori} convergence~\eqref{eq:a_priori_konv}.
	Then, the error estimator is uniformly bounded, i.e.,
	\begin{equation}
	\label{eq:estimator-boundedness}
		\sup_{\ell\in\N_0} \xi_\ell < \infty.
	\end{equation}
\end{lemma}

\begin{proof}
	Let $\ell \in \N_0$ be arbitrary.
	Then, stability~\eqref{assumption:stab}, together with the Young inequality, and reduction~\eqref{assumption:red}  yield that
	\begin{equation}
	\label{eq:boundedness-estimate}
		\xi_\ell^2
		=
		\xi_\ell(\TT_\ell \cap \TT_0)^2 + \xi_\ell(\TT_\ell \backslash \TT_0)^2
		\leq
		2 \, \xi_0^2 + 3 S\big( \norm{u_\ell - u_0}, \norm{z_\ell - z_0} \big)^2.
	\end{equation}
	From \textsl{a priori} convergence~\eqref{eq:a_priori_konv}, we further infer that the set $\set[\big]{\big( \norm{u_\ell - u_0}, \norm{z_\ell - z_0} \big)}{\ell \in \N_0} \subset \Rge^2$ is bounded and, thus, there is a compact set $K \subset \Rge^2$ that covers it.
	By assumption, $S$ is bounded on $K$.
	Together with~\eqref{eq:boundedness-estimate}, this yields a bound on $\xi_\ell$ that is independent of $\ell \in \N_0$.
	This concludes the proof.
\end{proof}

The following elementary result is essential for proving our main theorems.
Its proof is found, e.g., in \cite{gp2021}.

\begin{lemma}\label{lemma:folgenkonv}
    Let $(a_\ell)_{\ell \in \N_0}$ be a sequence with $a_\ell \geq 0$ for all $\ell \in \N_0$.
    Suppose that there exists $0 \leq \rho < 1$ and a sequence $(b_\ell)_{\ell\in\N_0}$ with $b_\ell \to 0$ as $\ell \to \infty$ such that
    \begin{equation*}
        a_{\ell+1}
        \leq
        \rho a_\ell + b_\ell
        \quad \text{for all } \ell \in \N_0.
    \end{equation*}
    Then, it follows that $a_\ell \to 0$ as $\ell \to \infty$.
\end{lemma}

The next three auxiliary results are implicitly found in~\cite{gp2021} or follow from the arguments given there.
For completeness of our presentation and the convenience of the reader, we state and prove the results in the form that is needed in the proofs of Theorems~\ref{thm:goafem_a} and~\ref{thm:goafem_b} below.
\begin{lemma}\label{lemma:gp1}
    Let $(\TT_\ell)_{\ell\in\N_0}$ be a sequence of triangulations such that $\TT_{\ell+1} \in \T(\TT_\ell)$ for all $\ell\in\N_0$.
    Recall from~\eqref{eq:def_t_infty} the set $\TT_\infty$ of all elements which are not refined after finitely many steps.
    Suppose that the refinement method satisfies~\eqref{assumption:refine1}--\eqref{assumption:refine2}.
    Then, there exists a subsequence $(\TT_{\ell_k})_{k\in\N_0}$ of $(\TT_\ell)_{\ell\in\N_0}$ such that
    \begin{equation}\label{eq:gp1_schnitt}
        \TT_{\ell_k} \cap \TT_{\ell_{k+n}} = \TT_{\ell_k} \cap \TT_\infty \qbox{for all }k\in\N_0, n\in\N.
    \end{equation}
    Suppose further that the error estimator $\xi_\ell$ satisfies reduction on refined elements~\eqref{assumption:red}.
    Then, \textsl{a priori} convergence~\eqref{eq:a_priori_konv} yields that
    \begin{align}\label{eq:gp1_konvergenz}
        \xi_{\ell_k} ( \TT_{\ell_k} \backslash \TT_{\ell_{k+1}} )
        \to 0
        \quad \text{as } k \to \infty.
    \end{align}
    In particular, it follows that
    $\xi_{\ell_k} ( \MM_{\ell_k} ) \to 0$
    as $k \to \infty$.
\end{lemma}

\begin{proof}
    First, we inductively define a strictly increasing sequence of indices $(\ell_k)_{k\in\N_0}$ such that
    \begin{equation}\label{eq:proofGp1_schnitt}
        \TT_{\ell_k}\cap\TT_{\ell_{k+1}}=\TT_{\ell_k}\cap \TT_\infty \qbox{for all } k\in\N_0.
    \end{equation}
    To this end, let $\ell_0:=0$ and suppose that the indices $\ell_0,\ldots,\ell_k$ are already defined in a way that satisfies~\eqref{eq:proofGp1_schnitt}.
    On the one hand, for each $T\in\TT_{\ell_k}\cap\TT_\infty$, the definition~\eqref{eq:def_t_infty} of $\TT_\infty$ yields an index $\ell_T$ with
    \begin{equation*}
        T\in\TT_\ell \quad \mbox{for all }\ell\geq\ell_T.
    \end{equation*}
    Since $\TT_{\ell_k}\cap\TT_\infty$ is finite, the index
    \begin{equation*}
        \hat\ell := \max\limits_{T\in\TT_{\ell_k}\cap\TT_\infty}\ell_T \in \N
    \end{equation*}
    is well-defined and, for all $\ell\geq\hat\ell$ and $T\in\TT_{\ell_k}\cap\TT_\infty$, it holds that $T\in\TT_{\ell_k}\cap\TT_\ell$, i.e.,
    \begin{equation*}
        \TT_{\ell_k}\cap\TT_{\ell}\supseteq\TT_{\ell_k}\cap \TT_\infty \qbox{for all } \ell\ge\hat\ell.
    \end{equation*}
    On the other hand, for fixed $T\in\TT_{\ell_k}\backslash\TT_\infty$, there exists an index $\ell'_T$ such that
    \begin{equation*}
        T\notin\TT_{\ell'} \qbox{for all } \ell'\geq\ell'_T.
    \end{equation*}
    Again, the index
    \begin{equation*}
        \hat\ell' := \max\limits_{T\in\TT_{\ell_k}\backslash\TT_\infty} \ell'_T \in \N
    \end{equation*}
    is well-defined and we conclude $T\in\TT_{\ell_k}\backslash\TT_{\ell}$ for any $\ell\geq\hat\ell'$. It follows that
    \begin{equation*}
        \TT_{\ell_k}\backslash\TT_{\ell} \supseteq \TT_{\ell_k}\backslash\TT_\infty \qqbox{and thus}
        \TT_{\ell_k}\cap\TT_{\ell}\subseteq\TT_{\ell_k}\cap \TT_\infty\qbox{for all }\ell\ge\hat\ell'.
    \end{equation*}
    Defining $\ell_{k+1}:=\max\{\hat\ell,\hat\ell'\}$, we obtain~\eqref{eq:proofGp1_schnitt}.
    To arrive at~\eqref{eq:gp1_schnitt}, note that from the definition~\eqref{eq:def_t_infty} of $\TT_\infty$ and~\eqref{assumption:refine2} it follows that
    \begin{equation*}
        \TT_{\ell_k}\cap\TT_\infty\subseteq\TT_{\ell_k}\cap\TT_{\ell_{k+n}} \qbox{for any } n\in\N.
    \end{equation*}
    On the other hand, any given $T\in\TT_{\ell_k}\cap\TT_{\ell_{k+n}}$ must be an element of $\TT_{\ell_{k+1}}$ and we conclude 
    \begin{equation*}
        T\in\TT_{\ell_{k}}\cap\TT_{\ell_{k+1}}=\TT_{\ell_k}\cap\TT_\infty,
    \end{equation*}
    which yields $\TT_{\ell_k}\cap\TT_{\ell_{k+n}} \supseteq \TTlk\cap\TT_\infty$ and, hence,~\eqref{eq:gp1_schnitt}.
    
    To prove convergence~\eqref{eq:gp1_konvergenz}, we first note that equation~\eqref{eq:proofGp1_schnitt} guarantees the inclusion
    \begin{equation*}
        \TT_{\ell_{k+1}} \cap \TT_{\ell_k} 
        =
        \TT_{\ell_{k+1}} \cap [ \TT_{\ell_{k+1}} \cap \TT_{\ell_k} ]
        \stackrel{\eqref{eq:proofGp1_schnitt}}{=}
        \TT_{\ell_{k+1}} \cap [ \TT_{\ell_k} \cap \TT_\infty ]
        \subseteq 
        \TT_{\ell_{k+1}} \cap \TT_\infty 
        \stackrel{\eqref{eq:proofGp1_schnitt}}{=} 
        \TT_{\ell_{k+2}} \cap \TT_{\ell_{k+1}}.
    \end{equation*}
    This implies that
    \begin{equation*}
        \TT_{\ell_{k+1}}\backslash \TT_{\ell_{k+2}} = 
        \TT_{\ell_{k+1}}\backslash [ \TT_{\ell_{k+2}} \cap \TT_{\ell_{k+1}} ] \subseteq 
        \TT_{\ell_{k+1}}\backslash [ \TT_{\ell_{k+1}} \cap \TT_{\ell_{k}} ] = \TT_{\ell_{k+1}}\backslash \TT_{\ell_{k}}
    \end{equation*}
    and we arrive at the estimate
    \begin{align*}
        \xi_{\ell_{k+1}} ( \TT_{\ell_{k+1}} \backslash \TT_{\ell_{k+2}} )^2 
        &\leq
        \xi_{\ell_{k+1}} ( \TT_{\ell_{k+1}} \backslash \TT_{\ell_{k}} )^2\\ 
        &\eqreff*{assumption:red}{\leq}
        \qred \, \xi_{\ell_{k}} ( \TT_{\ell_{k}} \backslash \TT_{\ell_{k+1}} )^2 
        + S \big( \norm{u_{\ell_{k+1}} - u_{\ell_k}}_\XX, \norm{z_{\ell_{k+1}} - z_{\ell_k}}_\XX \big)^2.
    \end{align*}
    \textsl{A priori} convergence~\eqref{eq:a_priori_konv} then allows to apply Lemma~\ref{lemma:folgenkonv} to
    \begin{equation*}
        a_k := \xi_{\ell_k}(\TT_{\ell_k}\backslash\TT_{\ell_{k+1}})^2
        ,\quad
        b_k:=S\big(\norm{u_{\ell_{k+1}}-u_{\ell_k}}_\XX, \norm{z_{\ell_{k+1}}-z_{\ell_k}}_\XX, \big)^2
        \qqbox{and}
        \rho:=\qred,
    \end{equation*}
    which yields the convergence~\eqref{eq:gp1_konvergenz}.
    
    Since $T\in\TT_{\ell_{k}}\backslash\TT_{\ell_{k}+1}$ implies that $T$ is refined in step $\ell_{k}$ and therefore $T\notin\TT_\infty$, it holds that
    \begin{equation*}
        \MM_{\ell_k} \subseteq \TT_{\ell_{k}} \backslash \TT_{\ell_{k}+1} \subseteq 
        \TT_{\ell_{k}} \backslash \TT_\infty = 
        \TT_{\ell_{k}} \backslash [ \TT_{\ell_k} \cap \TT_\infty ] \stackrel{\eqref{eq:proofGp1_schnitt}}{=} 
        \TT_{\ell_{k}} \backslash [ \TT_{\ell_{k+1}} \cap \TT_{\ell_{k}} ] = 
        \TT_{\ell_{k}} \backslash \TT_{\ell_{k+1}}.
    \end{equation*}
    Hence, we infer that
    \begin{equation*}
        0 \leq
        \xi_{\ell_k} ( \MM_{\ell_k} )
        \leq 
        \xi_{\ell_k} ( \TT_{\ell_{k}} \backslash \TT_{\ell_{k+1}} )
        \to 0
        \quad \text{as } k \to \infty,
    \end{equation*}
    which concludes the proof.
\end{proof}

\begin{lemma}\label{lemma:gp2}
    Let 
    $( \TT_\ell )_{ \ell \in \N_0 }$
    be a sequence of triangulations such that
    $\TT_{\ell + 1} \in \T( \TT_\ell )$ 
    for all $\ell \in \N_0$.
    Suppose that the refinement strategy satisfies~\eqref{assumption:refine1}--\eqref{assumption:refine2} and the error estimators $\eta_\ell$ and $\zeta_\ell$ satisfy stability~\eqref{assumption:stab} and reduction~\eqref{assumption:red}.
    Suppose further that 
    $( \TT_{\ell} )_{ \ell \in \N_0 }$
    allows for a subsequence $( \TT_{\ell_k} )_{ k \in \N_0 }$
    that satisfies~\eqref{eq:gp1_schnitt} as well as
    \begin{equation}\label{eq:gp2_conv}
	    \lim_{ \substack{ k \to \infty \\ \ell_k \ge \ell' } } 
    	    \big[
    	        \eta_{\ell_k} ( \TT_{\ell'} \cap \TT_\infty ) \,
    	        \zeta_{\ell_k} ( \TT_{\ell'} \cap \TT_\infty )
    	    \big]
	    = 0 
	    \quad \text{for all } \ell' \in \N_0.
    \end{equation}
    Then, \textsl{a priori} convergence~\eqref{eq:a_priori_konv} of 
	$(u_\ell)_{\ell\in\N_0}$ and $(z_\ell)_{\ell\in\N_0}$
	implies that
    \begin{align}\label{eq:gp2}
        \eta_\ell \, \zeta_\ell \to 0
        \quad \text{as } \ell \to \infty.
    \end{align}
\end{lemma}

\begin{proof}
	Let 
	$( \TT_{\ell_k} )_{ k \in \N_0 }$ 
	be a subsequence that satisfies~\eqref{eq:gp1_schnitt} as well as~\eqref{eq:gp2_conv}.
	We split the proof into two steps.
	
    \textbf{Step 1}: 
    We prove that
    $\eta_{\ell_k} \zeta_{\ell_k} \to 0$ 
    as $k \to \infty$.\newline
    To this end, note that we can assume, without loss of generality, the existence of 
    $K \in \N$ such that 
    $\eta_{\ell_k} \neq 0$ and $\zeta_{\ell_k} \neq 0$ 
    for all $k > K$, 
    since otherwise we can choose a subsequence that satisfies 
    $\eta_{\ell_{k_j}} \zeta_{\ell_{k_j}} = 0$ for all $j \in \N_0$.
    Furthermore, boundedness of the error estimators~\eqref{eq:estimator-boundedness} guarantees that
    \begin{equation}\label{eq:gp2_bound}
        \sup_{\ell \in \N_0} 
            ( \eta_\ell + \zeta_\ell )
        \le 
        \Cest 
        < 
        \infty.
    \end{equation}
    Let $\xi \in \{\eta, \zeta\}$. For all $\TT_\coarse \in \T$ and 	$\TT_\fine \in \T ( \TT_\coarse )$, reduction on refined elements~\eqref{assumption:red} yields that
    \begin{equation}
    \label{eq:proofGp2_eq1}
    \begin{split}
    	\xi_\fine^2
    	&=
    	\xi_\fine ( \TT_\fine \backslash \TT_\coarse )^2
    	+ \xi_\fine ( \TT_\fine \cap \TT_\coarse )^2\\
    	&\eqreff*{assumption:red}{\leq}
    	\qred \, \xi_\coarse ( \TT_\coarse \backslash \TT_\fine )^2
    	+ S \big( \norm{u_\fine - u_\coarse}_\XX, \norm{z_\fine - z_\coarse}_\XX \big)^2  
    	+ \xi_\fine ( \TT_\fine \cap \TT_\coarse )^2.
    \end{split}
    \end{equation}
    Applying this result to
    $\TT_\coarse := \TT_{\ell_{k}}$ and $\TT_\fine := \TT_{\ell_{k+n}}$
    with $k > K$ and $n \in \N$
    and defining
    \begin{equation*}
        S_{\fine,\coarse}
        :=
        S \big( \norm{u_\fine - u_\coarse}_\XX, \norm{z_\fine - z_\coarse}_\XX \big),
    \end{equation*}
    we obtain that
    \begin{align}
    \begin{split}
    \label{eq:gp2_red-ineq}
        \xi_{\ell_{k+n}}^2
        &\eqreff*{eq:proofGp2_eq1}{\leq}
        \qred \, \xi_{\ell_k} ( \TT_{\ell_k} \backslash \TT_{\ell_{k+n}} )^2 
        + \xi_{\ell_{k+n}} ( \TT_{\ell_{k+n}} \cap \TT_{\ell_k} )^2 
        + S_{\ell_{k+n},\ell_k}^2 
        \\
        &\eqreff*{eq:gp1_schnitt}{=}~
        \qred \, \xi_{\ell_k} ( \TT_{\ell_k} \backslash \TT_{\ell_{k+n}} )^2
        + \xi_{\ell_{k+n}} ( \TT_{\ell_{k}} \cap \TT_\infty )^2 
        + S_{\ell_{k+n},\ell_k}^2.
    \end{split}
    \end{align}
    The multiplication of the resulting inequality for $\zeta_{\ell_{k+n}}^2$ by $\eta_{\ell_{k+n}}^2$ yields that
    \begin{align*}
        \eta_{\ell_{k+n}}^2 \zeta_{\ell_{k+n}}^2
        & \eqreff*{eq:gp2_red-ineq}{\le}
        \eta_{\ell_{k+n}}^2 \, 
        \zeta_{\ell_{k+n}} ( \TT_{\ell_{k}} \cap \TT_\infty )^2
        + \eta_{\ell_{k+n}}^2
        \big[
            \qred \,
            \zeta_{\ell_k} ( \TT_{\ell_k} \backslash \TT_{\ell_{k+n}} )^2
            + S_{\ell_{k+n},\ell_k}^2
        \big]
        \\
        & \eqreff*{eq:gp2_bound}{\le}
        \eta_{\ell_{k+n}}^2 \, 
        \zeta_{\ell_{k+n}} ( \TT_{\ell_{k}} \cap \TT_\infty )^2
        + \Cest^2
        \big[
            \qred \,
            \zeta_{\ell_k} ( \TT_{\ell_k} \backslash \TT_{\ell_{k+n}} )^2
            + S_{\ell_{k+n},\ell_k}^2
        \big].
\end{align*}
The remaining primal estimator
$\eta_{\ell_{k+n}}$
can again be estimated by~\eqref{eq:gp2_red-ineq} to obtain
\begin{align*}
        \eta_{\ell_{k+n}}^2 \zeta_{\ell_{k+n}}^2
        & \eqreff*{eq:gp2_red-ineq}{\le} 
        \qred \, 
        \eta_{\ell_k} ( \TT_{\ell_k} \backslash \TT_{\ell_{k+n}} )^2 \,
        \zeta_{\ell_{k+n}} ( \TT_{\ell_k} \cap \TT_\infty )^2
        + \eta_{\ell_{k+n}} ( \TT_{\ell_k} \cap \TT_\infty )^2 \,
        \zeta_{\ell_{k+n}} ( \TT_{\ell_k} \cap \TT_\infty )^2
        \\
        & \qquad + S_{\ell_{k+n},\ell_k}^2 \, 
        \zeta_{\ell_k} ( \TT_{\ell_k} \cap \TT_\infty )^2
        + \Cest^2 \big[
            \qred \, 
            \zeta_{\ell_k} ( \TT_{\ell_k} \backslash \TT_{\ell_{k+n}} )^2
            + S_{\ell_{k+n},\ell_k}^2
        \big]
        \\
        &\eqreff*{eq:gp2_bound}{\le} 
        \eta_{\ell_{k+n}} ( \TT_{\ell_k} \cap \TT_\infty )^2 \,
        \zeta_{\ell_{k+n}} ( \TT_{\ell_k} \cap \TT_\infty )^2
        \\& \qquad 
        + \Cest^2 \big[
            \qred \, 
            \eta_{\ell_k} ( \TT_{\ell_k} \backslash \TT_{\ell_{k+n}} )^2
+ \qred \, 
            \zeta_{\ell_k} ( \TT_{\ell_k} \backslash \TT_{\ell_{k+n}} )^2
            + 2 S_{\ell_{k+n},\ell_k}^2
        \big].
    \end{align*}
    Let $\varepsilon>0$ be arbitrary.
    Lemma~\ref{lemma:gp1}, together with continuity of $S$ at $(0,0)$ 
    and \textsl{a~priori} convergence~\eqref{eq:a_priori_konv} of 
    $(u_\ell)_{\ell\in\N_0}$ as well as $(z_\ell)_{\ell\in\N_0}$
    yields an index $k_0 \in \N_0$ such that for all $n \in \N$ there holds
    \begin{align*}
        \qred \, \eta_{\ell_{k_0}} ( \TT_{\ell_{k_0}} \backslash \TT_{\ell_{k_0+n}} )^2
        + \qred \, \zeta_{\ell_{k_0}} ( \TT_{\ell_{k_0}} \backslash \TT_{\ell_{k_0+n}} )^2
        + 2 S_{\ell_{k_0+n},\ell_{k_0}}^2
        \le \varepsilon \, \Cest^{-2}.
    \end{align*}
    Similarly, the assumption~\eqref{eq:gp2_conv} guarantees the existence of an index $n_0 \in \N$ with
    \begin{align*}
        \eta_{\ell_{k_0+n}} ( \TT_{\ell_{k_0}} \cap \TT_\infty )^2 \,
        \zeta_{\ell_{k_0+n}} ( \TT_{\ell_{k_0}} \cap \TT_\infty )^2
        \le \varepsilon 
        \quad \text{ for all } n \ge n_0.
    \end{align*}
    For $k \ge k_0 + n_0$, it follows that    
    \begin{align*}
        \eta_{\ell_k}^2 \zeta_{\ell_k}^2
        =
        \eta_{\ell_{k_0 + (k-k_0)}}^2 \zeta_{\ell_{k_0 + (k-k_0)}}^2
        \le 2\varepsilon
    \end{align*}
    and we conclude $\eta_{\ell_k}\zeta_{\ell_k} \to 0$ as $k \to \infty$.
    
    \textbf{Step 2}: We prove that, for every $\varepsilon > 0$, there exists 
    $\ell_0 \in \N$ 
    such that 
    $\eta_\ell \, \zeta_\ell < \varepsilon$
    for all $\ell \ge \ell_0$. \newline
    Using~\eqref{eq:proofGp2_eq1} as well as $\qred<1$ and stability~\eqref{assumption:stab} of the error estimator
    $\xi \in \{\eta, \zeta\}$,
    we infer
    \begin{align}\begin{split}\label{eq:proofGp2_eq3}
        \xi_\fine^2 \,
        &\eqreff*{eq:proofGp2_eq1}{\leq}
        \qred \, 
        \xi_\coarse ( \TT_\coarse \backslash \TT_\fine )^2
        + \xi_\fine ( \TT_\fine \cap \TT_\coarse )^2
        + S_{\fine,\coarse}^2 
        \\
        & \eqreff*{assumption:stab}{\leq}
        \qred \, \xi_\coarse^2
        + \big[ 
            \xi_\coarse ( \TT_\fine \cap \TT_\coarse ) 
        	+ S_{\fine,\coarse}
        \big]^2
        + S_{\fine,\coarse}^2 
        \leq 2 ( \xi_\coarse + S_{\fine,\coarse} )^2.
    \end{split}\end{align}
    For given $\varepsilon>0$, there exists $k\in\N$ such that
    \begin{equation}
    \label{eq:gp2_step2-eps-bound}
        \eta_{\ell_{k}} \zeta_{\ell_{k}} 
        + S_{\ell,\ell_{k}}
        < \varepsilon
        \quad \text{for all } \ell \ge \ell_{k}.
    \end{equation}
    Thus, the application of~\eqref{eq:proofGp2_eq3} to $\TT_\coarse := \TT_{\ell_{k_j}}$ and $\TT_\fine := \TT_\ell$ yields that
    \begin{align*}
        \eta_\ell^2 \, \zeta_\ell^2
        & \le
        4 ( \eta_{\ell_{k}} + \varepsilon )^2
        ( \zeta_{\ell_{k}} + \varepsilon )^2
        \\
        &= 4 \big[
            \eta_{\ell_{k}} 
            ( \eta_{\ell_{k}} + 2\varepsilon )
            + \varepsilon^2 
        \big]
        \big[
            \zeta_{\ell_{k}} 
            ( \zeta_{\ell_{k}} + 2\varepsilon )
            + \varepsilon^2 
        \big]
        \\
        &= 4 \eta_{\ell_{k}} \zeta_{\ell_{k_j}}
        ( \eta_{\ell_{k}} + 2\varepsilon )
        ( \zeta_{\ell_{k}} + 2\varepsilon )
        + 4 \varepsilon^2
        \eta_{\ell_{k}}
        ( \eta_{\ell_{k}} + 2\varepsilon )
        + 4 \varepsilon^2 
        \zeta_{\ell_{k}} ( \zeta_{\ell_{k}} + 2\varepsilon )
        + 4 \varepsilon^4
        \\
        & \stackrel{\mathclap{
            \eqref{eq:gp2_bound},
            \eqref{eq:gp2_step2-eps-bound}
        }}
        {\le} \quad
        4 \varepsilon ( \Cest + 2\varepsilon )^2
        + 8 \varepsilon^2 \Cest ( \Cest + 2\varepsilon )
        + 4 \varepsilon^4.
    \end{align*}
    Since $\varepsilon > 0$ was arbitrary, this concludes the proof.
\end{proof}

\subsection{Proof of Theorem~\ref{thm:goafem_a}}

As in the proof of Theorem~\ref{thm:msv-result}, we may assume, without loss of generality, that Marking Strategy~\ref{alg:mark_a} is used; see Remark~\ref{rem:marking-strategies}.
The proof of Theorem~\ref{thm:goafem_a} is a straightforward application of the auxiliary results from Lemma~\ref{lemma:gp1}--\ref{lemma:gp2}:
    Lemma~\ref{lemma:gp1} yields a subsequence
    $( \TT_{\ell_k} )_{k \in \N_0}$
    that satisfies~\eqref{eq:gp1_schnitt} and~\eqref{eq:gp1_konvergenz}.
    Let $\ell' \in \N_0$ be arbitrary and $\ell_k \ge \ell'$.
    Then the marking criterion~\eqref{eq:mark_a} yields that
    \begin{align*}
        \eta_{\ell_k} ( \TT_{\ell'} \cap \TT_\infty )^2
        \zeta_{\ell_k} ( \TT_{\ell'} \cap \TT_\infty )^2
        & = \sum_{ T,T' \in \TT_{\ell'} \cap \TT_\infty }
        \eta_{\ell_k} (T)^2 \zeta_{\ell_k} (T')^2
        \\
        & \le \# [ \TT_{\ell'} \cap \TT_\infty ]
        \Big( \max_{ T \in \TT_{\ell_k} \backslash \MM_{\ell_k} }
        \eta_{\ell_k} (T)
        \max_{ T \in \TT_{\ell_k} \backslash \MM_{\ell_k} }
        \zeta_{\ell_k} (T) \Big)^2
        \\
        & \eqreff*{eq:mark_a}{\lesssim}
        V \big( 
            \eta_{\ell_k} ( \MM_{\ell_k} ) , 
            \zeta_{\ell_k} ( \MM_{\ell_k} )
        \big)^2
        \to 0 \quad \text{as } k \to \infty.
    \end{align*}
    Therefore, Lemma~\ref{lemma:gp2} yields that
    $\eta_\ell \, \zeta_\ell \to 0$
    as $\ell \to \infty$.
\qed

\subsection{Proof of Theorem~\ref{thm:goafem_b}}

The proof of Theorem~\ref{thm:goafem_b} makes heavy use of the structure provided in Marking Strategy~\ref{alg:mark_b} for the function $W$, which is more general in comparison to the special cases shown in~\cite{bet2011}.
	Note that this structure is inspired by Dörfler marking~\eqref{eq:doerfler}.
    We have to distinguish two cases:
    
    \textbf{Case 1:} There exists $\ell_0 \in \N_0$ such that $\eta_{\ell_0} \zeta_{\ell_0} = 0$.\newline
	According to Remark~\ref{rem:estimators_nonzero}, Algorithm~\ref{alg:goafem} is stopped in the $\ell_0$-th step.
	The sequence $(\TT_\ell)_{\ell \in \N_0} \subset \T$ as defined in Remark~\ref{rem:estimators_nonzero}  trivially satisfies the proposed convergence $\eta_\ell \zeta_\ell \to 0$ as $\ell \to \infty$.

	\textbf{Case 2:} For all $\ell \in \N_0$, there holds $\eta_\ell \zeta_\ell \neq 0$.\newline
	Let $\ell\in\N_0$ be arbitrary.
	The marking criterion~\eqref{eq:mark_b} of Marking Strategy~\ref{alg:mark_b} together with assumption~\eqref{eq:C-assumptions} yields that
    \begin{equation*}
    	0 < \theta
    	\eqreff*{eq:mark_b}{\leq}
    	W \bigg( \frac{\eta_\ell(\MM_\ell)^2}{\eta_\ell^2},
	   		\frac{\zeta_\ell(\MM_\ell)^2} {\zeta_\ell^2} \bigg)
	    \eqreff*{eq:C-assumptions}{\leq}
	    C_W \max \bigg\{ \frac{\eta_\ell(\MM_\ell)^2}{\eta_\ell^2},
	    	\frac{\zeta_\ell(\MM_\ell)^2} {\zeta_\ell^2} \bigg\}.
    \end{equation*}
    Without loss of generality, let the maximum be attained at $\eta_\ell$, i.e.,
    \begin{equation}\label{eq:bew_satz_c_ugl1}
        0 <
        \frac{\theta}{C_W}
        \leq
        \frac{\eta_\ell(\MM_\ell)^2}{\eta_\ell^2}
        \leq 1.
    \end{equation}
    To abbreviate notation, let $S_\ell:= S(\norm{u_{\ell+1} - u_\ell}_\XX, \norm{z_{\ell+1} - z_\ell}_\XX)$.
    Stability~\eqref{assumption:stab} and reduction~\eqref{assumption:red} of the error estimator $\eta_\ell$ together with the Young inequality for $\delta>0$ yield
    \begin{align*}
        \eta_{\ell+1}^2 &= \eta_{\ell+1}(\TT_{\ell+1}\cap\TT_\ell)^2 + \eta_{\ell+1}(\TT_{\ell+1}\backslash\TT_\ell)^2 \\
        &\stackrel{\mathclap{\eqref{assumption:stab}-\eqref{assumption:red}}}{\le} \quad
        (1+\delta) \, \eta_\ell(\TT_\ell\cap\TT_{\ell+1})^2 + \qred\eta_{\ell}(\TT_{\ell}\backslash\TT_{\ell+1})^2 + (2+\delta^{-1})S_\ell^2 \\
        &=
        (1+\delta) \, \eta_\ell^2 - (1+\delta-\qred)\eta_{\ell}(\TT_{\ell}\backslash\TT_{\ell+1})^2 + (2+\delta^{-1})S_\ell^2 \\
        &\le
        (1+\delta) \, \eta_\ell^2 - (1+\delta-\qred)\eta_{\ell}(\MM_\ell)^2 + (2+\delta^{-1})S_\ell^2 \\
        &\eqreff*{eq:bew_satz_c_ugl1}{\le}
        (1+\delta) \, \eta_\ell^2 - (1+\delta-\qred)\frac{\theta}{C_W}\eta_{\ell}^2 + (2+\delta^{-1})S_\ell^2
        =
        q \, \eta_\ell^2 + (2+\delta^{-1})S_\ell^2,
    \end{align*}
    where $q := (1+\delta) - (1+\delta-\qred)\theta/C_W$.
    Similarly, it follows for $\zeta_\ell$ and $\varepsilon>0$ that
    \begin{align*}
        \zeta_{\ell+1}^2
        &\le
        (1+\varepsilon) \, \zeta_\ell(\TT_\ell\cap\TT_{\ell+1})^2 
        + \qred \, \zeta_{\ell}(\TT_{\ell}\backslash\TT_{\ell+1})^2 + (2+\varepsilon^{-1})S_\ell^2 \\
        &\le
        (1+\varepsilon) \, \zeta_\ell^2 + (2+\varepsilon^{-1})S_\ell^2.
    \end{align*}
    Combining both estimates, we get
    \begin{align*}
        \eta_{\ell+1}^2\zeta_{\ell+1}^2
        &\le
        q(1+\varepsilon) \, \eta_{\ell}^2\zeta_{\ell}^2 + (2 + \delta^{-1}) (1+\varepsilon) S_\ell^2 \zeta_\ell^2 \\
        &\quad +  q (2+\varepsilon^{-1})S_\ell^2 \eta_\ell^2 + (2 + \delta^{-1})(2+\varepsilon^{-1})S_\ell^4.
    \end{align*}
    For all fixed $\delta,\varepsilon>0$, boundedness of the error estimators~\eqref{eq:estimator-boundedness} as well as \textsl{a priori} convergence~\eqref{eq:a_priori_konv} yield that
    \begin{equation*}
        b_\ell
        :=
        (2 + \delta^{-1})(1+\varepsilon) S_\ell^2 \zeta_\ell^2
        + q (2+\varepsilon^{-1})S_\ell^2 \eta_\ell^2
        + (2 + \delta^{-1}) (2+\varepsilon^{-1}) S_\ell^4
        \to 0
        \quad \text{as } \ell \to \infty.
    \end{equation*}
    Since $0 < \qred < 1$ and $0 < \theta / C_W \leq 1$, we may choose $\delta, \varepsilon > 0$ such that $q' := q(1+\varepsilon) \in (0,1)$.
    For $a_\ell := \eta_{\ell}^2\zeta_{\ell}^2$ for $\ell\in\N_0$, we then have that
    \begin{equation*}
    	a_{\ell+1} \leq q'a_\ell + b_\ell \qbox{for all } \ell \in \N_0.
    \end{equation*}
    By use of Lemma~\ref{lemma:folgenkonv}, we conclude that $a_\ell \to 0$ as $\ell \to \infty$.
\qed

\section{Examples}\label{sec:examples}

In this section, we underline the practicality of our main theorems.
To this end, we present four examples which fit into our settings and give numerical results for the first two.
Throughout, we use newest vertex bisection (NVB) for mesh-refinement; see, e.g., ~\cite{bdd2004,stevenson2008}.
All numerical examples are implemented using the Matlab AFEM package MooAFEM~\cite{mooafem}.

\subsection{Different marking strategies}\label{subsec:ms-example}

We consider the following example from~\cite{ms2009}.
Let $\Omega := (0,1)^2 \subset \R^2$ and let $\TT_0$ be chosen such that $\vec{f}, \vec{g} \in [L^2(\Omega)]^2$, given by
\begin{equation*}
	\vec{f}(x)
	:=
	\begin{cases}
		(-1,0) & \text{if } x_1 + x_2 < 1/2,\\
		(0,0) & \text{else,}
	\end{cases}
	\quad \text{and} \quad
	\vec{g}(x)
	:=
	\begin{cases}
	(1,0) & \text{if } x_1 + x_2 > 3/2,\\
	(0,0) & \text{else,}
	\end{cases}
\end{equation*}
are $\TT_0$-piecewise constant.
We suppose that the primal solution $u \in \XX := H^1_0 (\Omega)$ solves
\begin{equation}
\label{eq:model-ms}
	a(u,v) 
	:=
    \int_\Omega \nabla u \cdot \nabla v \d{x}
    =
    \int_\Omega \vec{f} \cdot \nabla v \d{x}
    =:
    F(v)
    \quad
    \text{for all  } v \in \XX,
\end{equation}
while the dual solution $z \in \XX$ solves
\begin{equation}
\label{eq:model-ms-dual}
	a(v,z)
	=
	\int_\Omega \vec{g} \cdot \nabla v \d{x}
	=:
	G(v)
	\quad
	\text{for all } v \in \XX.
\end{equation}
We are interested in the goal value $G(u) = a(u,z)$.
With $\XX_\coarse := \SS^p_0(\TT_\coarse)$, the discrete approximations of~\eqref{eq:model-ms}--\eqref{eq:model-ms-dual} then read
\begin{equation}
\label{eq:model-ms-discrete}
	a(u_\coarse,v_\coarse)
	=
	F(v_\coarse)
	\quad \text{and} \quad
	a(v_\coarse,z_\coarse)
	=
	G(v_\coarse)
	\qquad
	\text{for all  } v_\coarse \in \XX_\coarse.
\end{equation}
For~\eqref{eq:model-ms}--\eqref{eq:model-ms-discrete}, the residual \textsl{a~posteriori} error estimators read
\begin{equation}
\label{eq:model-ms-estimators}
\begin{split}
	\eta_\coarse(T)^2
	&:=
	h_T^2 \, \norm{\Delta u_\coarse}_{L^2(T)}^2
	+ h_T \, \norm{\jump{(\nabla u_\coarse - \vec{f} ) \cdot \normalvec}}_{L^2(\partial T \cap \Omega)}^2,\\
	\zeta_\coarse(T)^2
	&:=
	h_T^2 \, \norm{\Delta z_\coarse}_{L^2(T)}^2
	+ h_T \, \norm{\jump{(\nabla z_\coarse - \vec{g}) \cdot \normalvec}}_{L^2(\partial T \cap \Omega)}^2.
\end{split}
\end{equation}

\begin{theorem}[\cite{ms2009}]
	The error estimators $\eta_\coarse$ and $\zeta_\coarse$ from~\eqref{eq:model-ms-estimators} satisfy stability~\eqref{assumption:stab} and reduction~\eqref{assumption:red}.
	Furthermore, there holds \textsl{a~priori} convergence~\eqref{eq:a_priori_konv} as well as the goal error estimate~\eqref{eq:goal-error-estimate} with $G_\coarse := G(u_\coarse)$.
	\hfill $\square$
\end{theorem}

We solve this problem numerically by using Algorithm~\eqref{alg:goafem} with different (non-standard) marking strategies:
\begin{enumerate}[label={(\roman*)}]
	\item The general strategy from \cite{ms2009} (see Remark~\ref{rem:mark_a_stronger}) with the maximum marking criterion~\eqref{eq:maximum-marking} as considered in~\cite{mvy2020}.
	
	\item The general strategy from \cite{bet2011} (see Remark~\ref{rem:mark_a_stronger}) with the equidistribution marking criterion~\eqref{eq:equidistribution-marking}.
	
	\item Marking Strategy~\ref{alg:mark_b} with $W(x,y) := \max \{ \sin(\pi x/2), 2^y-1 \}$.
	
	\item Marking Strategy~\ref{alg:mark_b} with $W(x,y) := (x^q + y^q)^{1/q} / 2^{1/q}$ for $q = 10$.
\end{enumerate}
It is easy to see that both functions $W$ from~{\rm (iii)--(iv)} satisfy~\eqref{eq:C-assumptions}.
Thus, the following result follows from our main theorems in Section~\ref{sec:results}.

\begin{corollary}
	Let $(\TT_\ell)_{\ell \in \N_0} \subset \T$ be the output of Algorithm~\ref{alg:goafem} driven by the error estimators $\eta_\coarse$ and $\zeta_\coarse$ from~\eqref{eq:model-ms-estimators} and any of the marking criteria~{\rm (i)--(iv)}.
	Then, there holds convergence
	\begin{equation*}
		|G(u) - G(u_\ell)|
		\to 0
		\quad \text{as }
		\ell \to \infty.
	\end{equation*}
\end{corollary}

We plot the goal error estimator $\eta_\ell \, \zeta_\ell$ over $\sum_{k=0}^{\ell} \dim \XX_k$ in Figure~\ref{fig:ms-example}.
Provided that every step in Algorithm~\ref{alg:goafem} can be done in linear time with respect to the number of degrees of freedom $\dim \XX_\ell$, the quantity $\sum_{k=0}^{\ell} \dim \XX_k$ is the cumulative cost necessary to compute $\eta_\ell \, \zeta_\ell$.
We see that the summation-based marking strategies \textrm{(ii)} and \textrm{(iv)} do slightly better in this metric than the maximum-based marking strategies \textrm{(i)} and \textrm{(iii)}.

\begin{figure}
	\centering
	\resizebox{0.7\textwidth}{!}{\tikzstyle{plot}=[thick,mark options={scale=0.8}]%
\tikzstyle{reference}=[plot,dashed]%
\pgfplotstableread[col sep = comma]{plots/data/ms-p1-markA.dat}{\msPoneMarkA}%
\pgfplotstableread[col sep = comma]{plots/data/ms-p1-markB.dat}{\msPoneMarkB}%
\pgfplotstableread[col sep = comma]{plots/data/ms-p1-markC1.dat}{\msPoneMarkC}%
\pgfplotstableread[col sep = comma]{plots/data/ms-p1-markC2.dat}{\msPoneMarkD}%
\pgfplotstableread[col sep = comma]{plots/data/ms-p2-markA.dat}{\msPtwoMarkA}%
\pgfplotstableread[col sep = comma]{plots/data/ms-p2-markB.dat}{\msPtwoMarkB}%
\pgfplotstableread[col sep = comma]{plots/data/ms-p2-markC1.dat}{\msPtwoMarkC}%
\pgfplotstableread[col sep = comma]{plots/data/ms-p2-markC2.dat}{\msPtwoMarkD}%
\pgfplotstableread[col sep = comma]{plots/data/ms-p3-markA.dat}{\msPthreeMarkA}%
\pgfplotstableread[col sep = comma]{plots/data/ms-p3-markB.dat}{\msPthreeMarkB}%
\pgfplotstableread[col sep = comma]{plots/data/ms-p3-markC1.dat}{\msPthreeMarkC}%
\pgfplotstableread[col sep = comma]{plots/data/ms-p3-markC2.dat}{\msPthreeMarkD}%

\begin{tikzpicture}

\begin{loglogaxis}[%
	xlabel={$\sum_{k=0}^{\ell} \dim \XX_k$},%
	ylabel={$\eta_\ell \, \zeta_\ell$},%
	every axis plot/.append style={forget plot},%
	every tick label/.append style={font=\tiny},%
	legend pos=south west]

	\addplot[reference] table [x={cumulativeDofs}, y expr={5e-1/\thisrowno{1}}] {\msPoneMarkC} node at (axis description cs:0.9,0.5) {\tiny{$\alpha = -1$}};
	
	\addplot[plot,mark=x,color1] table [x={cumulativeDofs}, y={estimator}] {\msPoneMarkA};
	\addplot[plot,mark=x,color2] table [x={cumulativeDofs}, y={estimator}] {\msPoneMarkB};
	\addplot[plot,mark=x,color3] table [x={cumulativeDofs}, y={estimator}] {\msPoneMarkC};
	\addplot[plot,mark=x,color4] table [x={cumulativeDofs}, y={estimator}] {\msPoneMarkD};
	
	\addplot[reference] table [x={cumulativeDofs}, y expr={1e2/(\thisrowno{1}*\thisrowno{1})}] {\msPtwoMarkC} node at (axis description cs:0.9,0.27) {\tiny{$\alpha = -2$}};
	
	\addplot[plot,mark=o,color1] table [x={cumulativeDofs}, y={estimator}] {\msPtwoMarkA};
	\addplot[plot,mark=o,color2] table [x={cumulativeDofs}, y={estimator}] {\msPtwoMarkB};
	\addplot[plot,mark=o,color3] table [x={cumulativeDofs}, y={estimator}] {\msPtwoMarkC};
	\addplot[plot,mark=o,color4] table [x={cumulativeDofs}, y={estimator}] {\msPtwoMarkD};
	
	\addplot[reference] table [x={cumulativeDofs}, y expr={1e5/(\thisrowno{1}*\thisrowno{1}*\thisrowno{1})}] {\msPthreeMarkB} node at (axis description cs:0.9,0.05) {\tiny{$\alpha = -3$}};
	
	\addplot[plot,mark=+,color1] table [x={cumulativeDofs}, y={estimator}] {\msPthreeMarkA};
	\addplot[plot,mark=+,color2] table [x={cumulativeDofs}, y={estimator}] {\msPthreeMarkB};
	\addplot[plot,mark=+,color3] table [x={cumulativeDofs}, y={estimator}] {\msPthreeMarkC};
	\addplot[plot,mark=+,color4] table [x={cumulativeDofs}, y={estimator}] {\msPthreeMarkD};
	
	\addlegendimage{plot,no markers,color1}
	\addlegendimage{plot,no markers,color2}
	\addlegendimage{plot,no markers,color3}
	\addlegendimage{plot,no markers,color4}
	\addlegendimage{plot,only marks,mark=x}
	\addlegendimage{plot,only marks,mark=o}
	\addlegendimage{plot,only marks,mark=+}
	
	\addlegendentry{\rm (i)}
	\addlegendentry{\rm (ii)}
	\addlegendentry{\rm (iii)}
	\addlegendentry{\rm (iv)}
	\addlegendentry{$p=1$}
	\addlegendentry{$p=2$}
	\addlegendentry{$p=3$}
\end{loglogaxis}

\end{tikzpicture}}
	\caption{Goal error estimator $\eta_\ell \, \zeta_\ell$ over cumulative computational cost for the problem from Section~\ref{subsec:ms-example}.
	The colors correspond to the marking strategies~{\rm (i)--(iv)}, the markers correspond to different polynomial degrees of the FEM ansatz space $\XX_\ell$.}
	\label{fig:ms-example}
\end{figure}
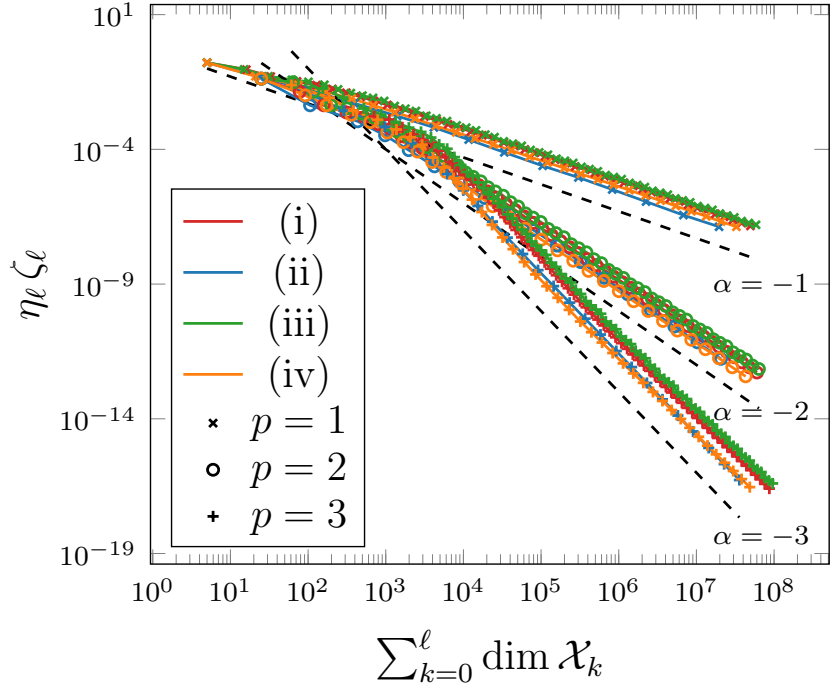

\subsection{Nonlinear goal functional}\label{subsec:bip-example}

We consider the L-shape domain $\Omega = (-1,1)^2 \backslash [-1,0]^2$ and the solution $u \in \XX := H^1_0(\Omega)$ of the Poisson equation
\begin{equation}
\label{eq:model-quadratic}
	a(u,v)
	:=
	\int_\Omega \nabla u \cdot \nabla v \d{x}
	=
	\int_\Omega f v \d{x}
	=:
	F(v)
	\quad
	\text{for all  } v \in \XX.
\end{equation}
We are interested in the quadratic goal value around $y = (0.5,0.5)$,
\begin{equation*}
\label{eq:goal-quadratic}
	G(u)
	:=
	\frac{1}{2} \int_\Omega \lambda_{y}(x) |\nabla u|^2 \d{x}
	\quad
	\text{with } \lambda_y(x) := \big[ 10^{-2} + \norm{x-y}^2 \big]^{-1},
\end{equation*}
where the linearized dual problem reads~\cite{bip2021}: Find $z \in \XX$ such that
\begin{equation}
\label{eq:model-quadratic-dual}
	a(v,z)
	=
	\int_\Omega \lambda_{y}(x) \nabla u \cdot \nabla v \d{x}
	=:
	b(u,v)
	\quad
	\text{for all  } v \in \XX.
\end{equation}
With $\XX_\coarse := \SS^p_0(\TT_\coarse)$, the discrete approximations of~\eqref{eq:model-quadratic}--\eqref{eq:model-quadratic-dual} then read
\begin{equation}
\label{eq:model-quadratic-discrete}
	a(u_\coarse,v_\coarse)
	=
	F(v_\coarse)
	\quad \text{and} \quad
	a(v_\coarse,z_\coarse)
	=
	b(u_\coarse, v_\coarse)
	\qquad
	\text{for all  } v_\coarse \in \XX_\coarse.
\end{equation}
The residual \textsl{a~posteriori} error estimators read
\begin{equation}
\label{eq:model-quadratic-estimators}
\begin{split}
	\mu_\coarse(T)^2
	&:=
	h_T^2 \, \norm{\Delta u_\coarse + f}_{L^2(T)}^2 + h_T \, \norm{\jump{\nabla u_\coarse \cdot \normalvec}}_{L^2(\partial T \cap \Omega)}^2,\\
	\nu_\coarse(T)^2
	&:=
	h_T^2 \, \norm{\Delta z_\coarse + \lambda_y(x) \Delta u_\coarse + \nabla \lambda_y(x) \cdot \nabla u_\coarse}_{L^2(T)}^2\\
	& \quad + h_T \, \norm{\jump{(\nabla z_\coarse + \lambda_y(x) \nabla u_\coarse) \cdot \normalvec}}_{L^2(\partial T \cap \Omega)}^2.
\end{split}
\end{equation}
The work~\cite{bip2021} shows that there holds the goal error estimate
\begin{equation*}
	|G(u) - G(u_\coarse)|
	\lesssim
	\mu_\coarse \, \big[ \mu_\coarse^2 + \nu_\coarse^2 \big]^{1/2}
	\leq
	\mu_\coarse^2 + \nu_\coarse^2.
\end{equation*}

\begin{theorem}[\cite{bip2021}]
\label{th:axioms-quadratic}
	With $\mu_\coarse$ and $\nu_\coarse$ from~\eqref{eq:model-quadratic-estimators}, both choices of error estimators
	\begin{itemize}
		\item $\eta_\coarse := \mu_\coarse$ and $\zeta_\coarse := \big[ \mu_\coarse^2 + \nu_\coarse^2 \big]^{1/2}$ as well as
		
		\item $\eta_\coarse := \zeta_\coarse := \big[ \mu_\coarse^2 + \nu_\coarse^2 \big]^{1/2}$
	\end{itemize}
	satisfy stability~\eqref{assumption:stab}, reduction~\eqref{assumption:red}, and the goal error estimate~\eqref{eq:goal-error-estimate} with $G_\coarse := G(u_\coarse)$.
	Furthermore, there holds \textsl{a~priori} convergence~\eqref{eq:a_priori_konv}.
	\hfill $\square$
\end{theorem}

This allows to conclude convergence of both proposed schemes.

\begin{corollary}
	Let $(\TT_\ell)_{\ell \in \N_0} \subset \T$ be the output of Algorithm~\ref{alg:goafem} driven by any of the error estimators $\eta_\coarse$ and $\zeta_\coarse$ from Theorem~\ref{th:axioms-quadratic} and any of the Marking Strategies~\ref{alg:mark_a}--\ref{alg:mark_b}.
	Then, there holds convergence
	\begin{equation*}
		|G(u) - G(u_\ell)|
		\to 0
		\quad \text{as }
		\ell \to \infty.
	\end{equation*}
\end{corollary}

For problem~\eqref{eq:model-quadratic} with $f = 1$, numerical results were obtained by Algorithm~\ref{alg:goafem} with Marking Strategy~\ref{alg:mark_b} and $W(x,y) = (x+y)/2$ (the strategy presented in~\cite{bet2011}) for both choices of $\eta_\ell$ and $\zeta_\ell$; see Figure~\ref{fig:bip-example}.
Examples of meshes after some steps of the adaptive algorithm with different polynomial degrees are given in Figure~\ref{fig:meshes}.
For higher polynomial order, the refinement towards (re-entrant) corners is much stronger than for the lowest order case $p=1$.

\begin{figure}
	\centering
	\resizebox{0.45\textwidth}{!}{\tikzstyle{plot}=[thick]%
\tikzstyle{reference}=[plot,dashed]%
\pgfplotstableread[col sep = comma]{plots/data/bip-p1-combineSum.dat}{\bipPoneSum}%
\pgfplotstableread[col sep = comma]{plots/data/bip-p2-combineSum.dat}{\bipPtwoSum}%
\pgfplotstableread[col sep = comma]{plots/data/bip-p3-combineSum.dat}{\bipPthreeSum}%
\begin{tikzpicture}

\begin{loglogaxis}[%
	xlabel={$\sum_{k=0}^{\ell} \dim \XX_k$},%
	ylabel={$\eta_\ell \, \zeta_\ell = \mu_\ell^2 + \nu_\ell^2$},%
	every tick label/.append style={font=\tiny},%
	legend pos=south west]

	\addplot[reference, forget plot] table [x={cumulativeDofs}, y expr={3e3/\thisrowno{1}}] {\bipPoneSum} node[rotate=-17] at (axis description cs:0.5,0.81) {\tiny{$\alpha = -1$}};
	\addplot[reference, forget plot] table [x={cumulativeDofs}, y expr={1e6/(\thisrowno{1}*\thisrowno{1}*\thisrowno{1})}] {\bipPoneSum} node[rotate=-42] at (axis description cs:0.4,0.55) {\tiny{$\alpha = -3$}};
	
	\addplot[plot,mark=x,color1] table [x={cumulativeDofs}, y={estimator}] {\bipPoneSum};
	\addplot[plot,mark=o,color2] table [x={cumulativeDofs}, y={estimator}] {\bipPtwoSum};	
	\addplot[plot,mark=diamond,color3] table [x={cumulativeDofs}, y={estimator}] {\bipPthreeSum};
	
	\addlegendentry{$p=1$}
	\addlegendentry{$p=2$}
	\addlegendentry{$p=3$}
\end{loglogaxis}

\end{tikzpicture}}%
	\quad%
	\resizebox{0.45\textwidth}{!}{\tikzstyle{plot}=[thick]%
\tikzstyle{reference}=[plot,dashed]%
\pgfplotstableread[col sep = comma]{plots/data/bip-p1-combineProduct.dat}{\bipPoneSum}%
\pgfplotstableread[col sep = comma]{plots/data/bip-p2-combineProduct.dat}{\bipPtwoSum}%
\pgfplotstableread[col sep = comma]{plots/data/bip-p3-combineProduct.dat}{\bipPthreeSum}%
\begin{tikzpicture}

\begin{loglogaxis}[%
	xlabel={$\sum_{k=0}^{\ell} \dim \XX_k$},%
	ylabel={$\eta_\ell \, \zeta_\ell = \mu_\ell \, [\mu_\ell^2 + \nu_\ell^2]^{1/2}$},%
	every tick label/.append style={font=\tiny},%
	legend pos=south west]

	\addplot[reference, forget plot] table [x={cumulativeDofs}, y expr={3e3/\thisrowno{1}}] {\bipPoneSum} node[rotate=-15] at (axis description cs:0.5,0.8) {\tiny{$\alpha = -1$}};
	
	\addplot[reference, forget plot] table [x={cumulativeDofs}, y expr={1e6/(\thisrowno{1}*\thisrowno{1}*\thisrowno{1})}] {\bipPoneSum} node[rotate=-43] at (axis description cs:0.39,0.55) {\tiny{$\alpha = -3$}};
	
	\addplot[plot,mark=x,color1] table [x={cumulativeDofs}, y={estimator}] {\bipPoneSum};
	\addplot[plot,mark=o,color2] table [x={cumulativeDofs}, y={estimator}] {\bipPtwoSum};	
	\addplot[plot,mark=diamond,color3] table [x={cumulativeDofs}, y={estimator}] {\bipPthreeSum};
	
	\addlegendentry{$p=1$}
	\addlegendentry{$p=2$}
	\addlegendentry{$p=3$}
\end{loglogaxis}

\end{tikzpicture}}
	\caption{Goal error estimator $\eta_\ell \, \zeta_\ell$ over cumulative computational cost for the problem from Section~\ref{subsec:bip-example} with different polynomial degrees of the FEM ansatz space $\XX_\ell$.
	Left: Goal error bound with sum structure.
	Right: Sharper goal error bound with product structure.}
	\label{fig:bip-example}
\end{figure}
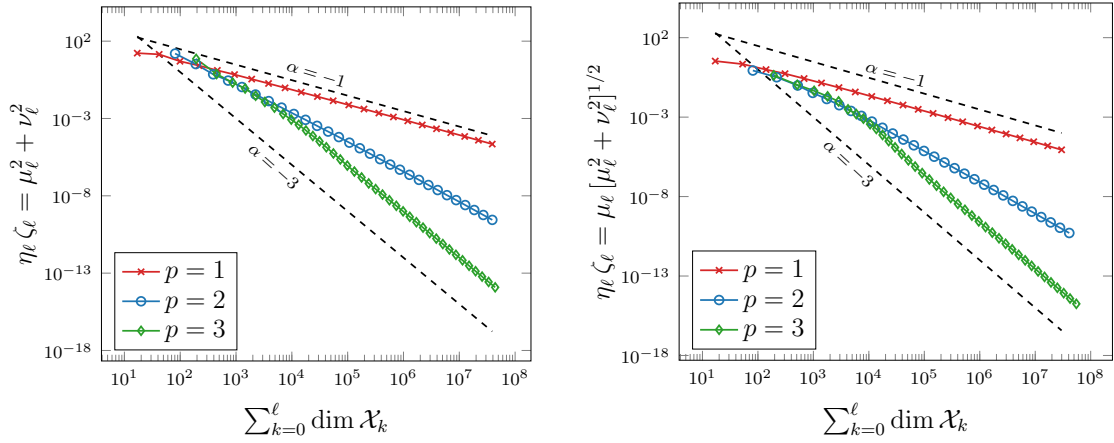

\begin{figure}
	\centering
	\includegraphics[width=0.3\textwidth]{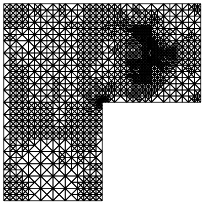}%
	\hspace{0.1\textwidth}%
	\includegraphics[width=0.3\textwidth]{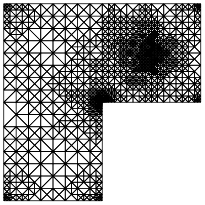}
	\caption{Meshes after several steps of the adaptive algorithm driven by $\eta_\ell \, \zeta_\ell = \mu_\ell \, [\mu_\ell^2 + \nu_\ell^2]^{1/2}$ for the problem from Section~\ref{subsec:bip-example}.
	The weight of the goal functional is centered at $(0.5, 0.5)$.
	Left: Computation with $p=1$, $\#\TT_6 = 3093$.
	Right: Computation with $p=3$, $\#\TT_{14} = 2821$.}
	\label{fig:meshes}
\end{figure}

\subsection{Semilinear problem}

The works~\cite{hpz15,bbi+2022} consider $\Omega \subset \R^d$ with
\begin{equation}
\label{eq:model-hpz}
	a(u,v) 
	:=
	\int_\Omega A \nabla u \cdot \nabla v + b(u) v \d{x}
	= \int_\Omega f v \d{x}
	=:
	F(v)
	\quad
	\text{for all  } v \in \XX,
\end{equation}
where $A \in [L^\infty(\Omega)]^{d \times d}$ is a uniformly positive definite matrix, $b \colon \R \to \R$ is a smooth and monotone function, i.e., $b'(s) \geq 0$ for all $s \in \R$, and $f \in L^2(\Omega)$.
The dual solution $z \in H_0^1(\Omega)$ is given by
\begin{equation}
\label{eq:model-hpz-dual}
	a'(u; v, z)
	:=
	\int_\Omega A \nabla v \cdot \nabla z + b'(u)vz \d{x}
	=
	\int_\Omega g v \d{x}
	=:
	G(v),
	\quad
	\text{for all } v \in \XX,
\end{equation}
where $g \in L^2(\Omega)$, and its discretization $z_\coarse \in \XX_\coarse$ satisfies
\begin{equation*}
	a'(u_\coarse; v_\coarse, z_\coarse) = G(v_\coarse)
	\quad
	\text{for all } v_\coarse \in \XX_\coarse.
\end{equation*}
The residual error estimators in this case read
\begin{equation}
\label{eq:model-semilinear-estimators}
\begin{split}
	\mu_\coarse(T)^2
	&:=
	h_T^2 \, \norm{\div(A \nabla u_\coarse) - b(u_\coarse) + f}_{L^2(T)}^2
	+ h_T \, \norm{\jump{A \nabla u_\coarse \cdot \normalvec}}_{L^2(\partial T \cap \Omega)}^2,\\
	\nu_\coarse(T)^2
	&:=
	h_T^2 \, \norm{\div(A \nabla z_\coarse) - b'(u_\coarse) z_\coarse + g}_{L^2(T)}^2
	+ h_T \, \norm{\jump{A \nabla z_\coarse \cdot \normalvec}}_{L^2(\partial T \cap \Omega)}^2.
\end{split}
\end{equation}
For these estimators, the work~\cite{bbi+2022} proves an analogue to Theorem~\ref{th:axioms-quadratic}.
Thus, our framework allows to conclude convergence of the goal error also in this case.
\begin{corollary}
	Let $(\TT_\ell)_{\ell \in \N_0} \subset \T$ be the output of Algorithm~\ref{alg:goafem} driven by any of the error estimators $\eta_\coarse$ and $\zeta_\coarse$ from Theorem~\ref{th:axioms-quadratic} with $\mu_\coarse$ and $\nu_\coarse$ from~\eqref{eq:model-semilinear-estimators}, and any of the Marking Strategies~\ref{alg:mark_a}--\ref{alg:mark_b}.
	Then, there holds convergence
	\begin{equation}
	\label{eq:goal-plain-convergence}
		|G(u) - G(u_\ell)|
		\to 0
		\quad \text{as }
		\ell \to \infty.
	\end{equation}
\end{corollary}

\begin{remark}
	In the work~\cite{hpz15}, the goal error estimate~\eqref{eq:goal-error-estimate} is shown under the assumption that the primal solutions $u, u_\coarse$ satisfy $L^\infty(\Omega)$ bounds, i.e., for some constant $C \in \R$ it holds that $-C \leq u, u_\coarse \leq C < \infty$ a.e.\ in $\Omega$.
	Furthermore, Marking Strategy~\ref{alg:mark_a} with $\MM_\ell = \MM_\ell^\eta \cup \MM_\ell^\zeta$ for all $\ell \in \N_0$ is used.
	According to Remark~\ref{rem:mark_a_stronger} and Theorem~\ref{thm:goafem_a}, this strategy leads to plain convergence of the goal error~\eqref{eq:goal-plain-convergence}.
	
	However, the particular choice of marking strategy in~\cite{hpz15} (more specifically, the choice $\MM_\ell = \MM_\ell^\eta \cup \MM_\ell^\zeta$) prevents convergence with optimal algebraic rates, which can be shown for the strategy from~\cite{bbi+2022}, even without assuming any (discrete) $L^\infty(\Omega)$ bounds.
\end{remark}


\renewcommand*{\bibfont}{\small}
\printbibliography

\end{document}